\documentclass{amsart}
\usepackage{amssymb}
\usepackage[mathscr]{eucal}
\usepackage{xspace}
\newcommand{\R}{\mathbb{R}}

\newcommand{\Q}{\mathbb{Q}}
\newcommand{\DD}{\mathcal{D}}
\newcommand{\Z}{\mathbb{Z}}
\newcommand{\N}{\mathbb{N}}

\newcommand{\LL}{\mathcal{L}}
\newcommand{\GG}{\mathcal{G}}

\newcommand{\F}{{\mathcal F}}

\newcommand{\po}{\partial}

\newcommand{\wto}{\rightharpoonup} 
\newcommand{\ve}{\varepsilon}

\newcommand{\la}{\langle}
\newcommand{\ra}{\rangle}

\newcommand{\loc}{{\text{\rm loc}}}
\newcommand{\X}{\times}

\renewcommand{\d}{\delta}
\renewcommand{\l}{\lambda}

\renewcommand{\a}{\alpha}
\renewcommand{\b}{\beta}

\newcommand{\s}{\sigma}
\newcommand{\g}{\gamma} 
\newcommand{\z}{\zeta}
\renewcommand{\k}{\kappa}
\newcommand{\sgn}{\text{\rm sgn}}

\newcommand{\om}{\omega}

\newcommand{\supp}{\text{\rm supp}\,}

\newcommand{\M}{{\mathcal M}}
\newcommand{\V}{{\mathcal V}}

\renewcommand{\div}{\text{\rm div}\,}

\renewcommand{\supp}{\text{\rm supp}\,}

\newcommand{\I}{\mathcal{I}}

\newcommand{\bff}{{\mathbf f}}

\newcommand{\bfA}{{\mathbf A}}

\newcommand{\A}{{\mathcal A}}

\newcommand{\bfq}{{\mathbf q}}
\newcommand{\bfR}{{\mathbf R}}

\newcommand{\BAP}{\operatorname{BAP}}
\newcommand{\AP}{\operatorname{AP}}
\newcommand{\SAP}{\operatorname{SAP}}

\newcommand{\GB}{\mathbb G}
\newcommand{\TT}{{\mathbb T}}

\newcommand{\fbf}{\mathbf{f}}

\newcommand{\Abf}{\mathbf{A}}
\newcommand{\Bbf}{\mathbf{B}}

\renewcommand{\S}{{\mathcal S}}

\newcommand{\mm}{{\mathfrak m}}

\newcommand{\Me}{\operatorname{M}}
\newcommand{\Sp}{\operatorname{Sp}}
\newcommand{\Gr}{\operatorname{Gr}}

\newcommand{\medint}{{\mbox{\vrule height3.5pt depth-2.8pt
          width4pt}\mkern-13mu\int\nolimits}}
\newcommand{\Medint}{\mkern12mu\mbox{\vrule height4pt
         depth-3.2pt
          width5pt}\mkern-16.5mu\int\nolimits}

\newcommand{\intl}{\int\limits}
\newcommand{\iintl}{\iint\limits}

\theoremstyle{plain}
\newtheorem{theorem}{Theorem}[section]

\newtheorem{lemma}{Lemma}[section]
\newtheorem{proposition}{Proposition}[section]
\theoremstyle{definition}
\newtheorem{definition}{Definition}[section]
\theoremstyle{remark}

\newtheorem{remark}{Remark}[section]
\numberwithin{equation}{section}

\begin{document}

\title[Decay of Almost Periodic Solutions]
{Asymptotic Decay of Besicovitch Almost Periodic Entropy Solutions to Anisotropic   Degenerate Parabolic-Hyperbolic Equations}
\author{Hermano Frid} 

\address{Instituto de Matem\'atica Pura e Aplicada - IMPA\\
         Estrada Dona Castorina, 110 \\
         Rio de Janeiro, RJ 22460-320, Brazil}
\email{hermano@impa.br}

\author{Yachun Li}

\address{School of Mathematical Sciences, MOE-LSC, and SHL-MAC,  Shanghai Jiao Tong University\\ Shanghai 200240, P.R.~China}

\email{ycli@sjtu.edu.cn}


\date{}
\thanks{}

\begin{abstract}
We prove the well-posedness and the asymptotic decay to the mean value of Besicovitch almost periodic solutions to  nonlinear  anisotropic degenerate parabolic-hyperbolic equations. 

\end{abstract}

\maketitle

\section{Introduction}\label{S:1}

We address the problem of the decay to the mean-value of $L^\infty$ Besicovitch almost periodic solutions to nonlinear  anisotropic   degenerate parabolic-hyperbolic equations. 
Consider the  Cauchy problem 
\begin{align}
& \po_tu+\nabla\cdot \bff(u)=\nabla^2: \bfA(u),\qquad x\in\R^d,\quad t>0,\label{e1.1}\\
&u(0,x)=u_0, \quad x\in\R^d,\label{e1.2}
\end{align}
where $\bff=(f_1,\cdots,f_d)$, $\bfA(u)=(A_{ij}(u))_{i,j=1}^d$, with $f_i(u), A_{ij}(u):\R\to\R$ smooth functions. We denote $\nabla^2:\bfA(u):=\sum_{i,j=1}^d \po_{x_ix_j}^2A_{ij}(u)$.   $\bfA(u)$ is a symmetric matrix such that its derivative 
$\bfA'(u)=A(u)=(a_{ij}(u))_{i,j=1}^d$, $a_{ij}(u)=A_{ij}'(u)$, is a non-negative matrix. In particular,  we may write 
\begin{equation}\label{e1.3}
a_{ij}(u)=\sum_{k=1}^d \s_{ik}(u)\s_{jk}(u),
\end{equation}
with $\s_{ij}(u):\R\to\R$ smooth functions, that is,  $(\s_{ij}(u))_{ij=1}^d$ is the square root of $A(u)$.  

We assume to begin with that $u_0\in L^\infty(\R^d)$.

In this paper, we are concerned with the large-time behavior of entropy solutions of \eqref{e1.1}-\eqref{e1.2} with initial function $u_0$ satisfying
\begin{equation}\label{e1.5}
u_0\in L^\infty(\R^d)\cap \BAP(\R^d).
\end{equation}
Here, $\BAP(\R^d)$ denotes the space of the Besicovitch almost periodic functions (with exponent $p=1$), which can be defined as the completion of the space of trigonometric polynomials, i.e., finite sums $\sum_{\l}a_\l e^{2\pi i\l\cdot x}$ ($i=\sqrt{-1}$ is the purely imaginary unity) under the semi-norm
$$
N_1(g):= \limsup_{R\to\infty} \frac{1}{R^d}\int_{\I_R}|g(x)|\,dx,
$$
where, for $R>0$,
$$
\I_R:= \{x\in\R^d\,:\,|x|_\infty:=\max_{i=1,\cdots,d}|x_i|\le R/2\}.
$$
 We observe that the semi-norm $N_1$ is indeed a norm over the trigonometric polynomials, so the referred completion through it is a well defined Banach space. Equivalently, the space $\BAP(\R^d)$ is also the completion through $N_1$ of the space of uniform (or Bohr) almost periodic functions, $\AP(\R^d)$, which is defined as the closure in the $\sup$-norm of the trigonometric polynomials. 
 
We begin by stating the definition of entropy solution for \eqref{e1.1}-\eqref{e1.2}, which is  motivated by  \cite{CP}. 

\begin{definition}\label{D:1.1}  An entropy solution for \eqref{e1.1}-\eqref{e1.2}, with $u_0\in L^\infty(\R^d)$,  is a function $u(t,x)\in L^\infty((0,\infty)\X\R^d)$ such that 
\begin{enumerate}
\item[(i)] (Regularity) For any $R>0$ and any $k=1,\cdots, d$,   we have
\begin{equation}\label{e1.6}
\sum_{i=1}^d\po_{x_i}\b_{ik}(u) \in L_\loc^2((0,\infty)\X \I_{R}),\ \ \text{for $\b_{ik}(u)=\int^u\s_{ik}(v)\,dv$}.
\end{equation}

\item[(ii)] (Chain Rule)  For any function $\psi\in C(\R)$ and any $k=1,\cdots, d$, the following chain rule holds:
\begin{multline}\label{e1.8}
\sum_{i=1}^d\po_{x_i}\b_{ik}^\psi(u)=\psi(u)\sum_{i=1}^d\po_{x_i}\b_{ik}(u)\in L_\loc^2((0,\infty)\X \I_{R}),\\ \quad\text{for any $R>0$ and $(\b_{ik}^\psi)'=\psi\b_{ik}'$}.
\end{multline}
\item[(iii)] (Entropy Inequality) For any convex $C^2$ function $\eta:\R\to\R$, and ${\bfq}'(u)=\eta'(u)\bff(u)$, $r_{ij}'(u)=\eta'(u)a_{ij}(u)$, we have
\begin{equation}\label{e1.9}
\po_t\eta(u)+\nabla\cdot \bfq(u)-\sum_{ij=1}^d \po_{x_ix_j}^2r_{ij}(u)
\le -\eta''(u)\sum_{k=1}^d\left(\sum_{i=1}^d\po_{x_i}\b_{ik}(u)\right)^2.
\end{equation}
\item[(iv)] (Initial Condition) For any $R>0$,
\begin{equation}\label{e1.10}
\lim_{t\to0+}\intl_{\I_R}|u(t,x)-u_0(x)|\,dx=0,
\end{equation}
\end{enumerate}
\end{definition}


\begin{remark}\label{R:1.2} It is easy to verify that the chain rule \eqref{e1.8} guarantees, in particular, that the vector fields $\nabla\cdot \bfA(u)$, $\nabla\cdot \left(\sgn(u-k)(\bfA(u)-\bfA(k))\right)$, $k\in\R$, and, more generally, $\nabla\cdot \bfR(u)$, $\bfR(u)=(r_{ij}(u))_{i,j=1}^d$, $r_{ij}'(u)=\eta'(u)a_{ij}(u)$, for any smooth entropy $\eta$, belong to $L^2((0,\infty)\X \I_{R})$, for any $R>0$. 
\end{remark}

For any $g\in\BAP(\R^d)$, its mean value $\Me(g)$, defined by
$$
\Me(g):=\lim_{R\to\infty} R^{-d}\int_{\I_R} g(x)\,dx,
$$
exists (see, e.g., \cite{B}). The mean value $\Me(g)$ is also denoted by $\medint_{\R^d}g\,dx$.  Also, the  Bohr-Fourier coefficients of $g\in\BAP(\R^d)$ 
$$
a_\l= \Me(g e^{-2\pi i\l\cdot x}),
$$ 
are well defined and we have that the spectrum of $g$, defined by
$$
\Sp(g):=\{ \l\in\R^d \,:\, a_{\l}\ne 0\},
$$
is at most countable (see, e.g., \cite{B}). We denote by  $\Gr(g)$ the smallest additive subgroup of $\R^d$ containing $\Sp(g)$. we now state the main result of this paper.

\begin{theorem}\label{T:1.1}  For any $u_0\in L^\infty(\R^d)$, there exists a unique weak entropy solution $u(t,x)$ of \eqref{e1.1}-\eqref{e1.2}. If $u_0$ satisfies \eqref{e1.5}, then 
\begin{equation}\label{e1.12}
u \in  C([0,\infty),\BAP(\R^d))\bigcap L^\infty(\R_+^{d+1}). 
\end{equation}
Moreover, if \eqref{e1.1} satisfies the following  non-degeneracy condition, where $a(\xi):=\bff'(\xi):$ 
for any $\d>0$, 
\begin{equation}\label{e1.4'}
\sup_{|\tau|+|\k|\ge \d}\int_{|\xi|\le \|u_0\|_\infty} \frac{\ell\,d\xi}{\ell+|\tau+a(\xi)\cdot\k|^2+(\k^\top A(\xi)\k)^2}:=\om_\d(\ell)\underset{\ell\to0+}{\to}0,
\end{equation}
then,
\begin{equation}\label{e1.13}
\lim_{t\to+\infty}\Me(|u(t,\cdot)-\Me(u_0)|)=0.
\end{equation}
\end{theorem}

\bigskip
\begin{remark}\label{R:1.3}  We remark that condition \eqref{e1.4'} is equivalent to the following condition:  
for any $(\tau,\k)\in\R^{d+1}$ with $\tau^2+|\k|^2=1$, 
\begin{equation}\label{e1.4}
\LL^1\{\xi\in\R\,:\,|\xi|\le\|u_0\|_\infty,\; \tau+a(\xi)\cdot\k=0,\;\k^\top A(\xi)\k=0\}=0. 
\end{equation}

Indeed, first we see that if \eqref{e1.4} does not hold, then, for some $(\tau,\k)$, with $\tau^2+|\k|^2=1$, $\LL^1\{\xi\in\R\,:\,|\xi|\le\|u_0\|_\infty,\; \tau+a(\xi)\cdot\k=0,\;\k^\top A(\xi)\k=0\}>0$. 
Therefore, for such $(\tau,\k)$, the integrand of the integral in \eqref{e1.4'} equals 1 in a fixed set of positive measure, for all $\ell>0$. Hence, \eqref{e1.4'} does not hold as well.  

Now, assume that \eqref{e1.4} holds. We first observe that the $\sup$ must be assumed for $|\tau|+|\k|=\d$, since the integrand decreases when $|\tau|+|\k|$ increases, which is easily seen by writing the integrand in terms of $\bar\tau=\tau/(|\tau|+|\k|)$, $\bar\k=\k/(|\tau|+|\k|)$ and $r=|\tau|+|\k|$. Let $I_\ell(\tau,\k)$ denote the integral in \eqref{e1.4'}. The functions $I_\ell(\tau,\k)$ are clearly continuous on $\S_\d^d:= \{(\tau,\k)\,:\, |\tau|+|\k|=\delta\}$. Moreover, condition \eqref{e1.4} implies that, for each $(\tau,\k)\in\S_\d^d$,  $I_\ell(\tau,\k)$ decreases to 0 as $\ell\to0+$. Therefore, Dini's theorem implies that $I_\ell\to0$, as $\ell\to0+$ uniformly on $\S_\d^d$, which implies \eqref{e1.4'}.   
\end{remark}

\bigskip
There is a large literature related with degenerate parabolic equations, being the first important contribution by  Vol'pert and Hudjaev in \cite{VH}. Uniqueness for the homogeneous Dirichlet problem, for the isotropic case, was only achieved  many years later by Carrillo in \cite{Ca}, using an extension of Kruzhkov's doubling of variables method \cite{Kr}. The result in \cite{Ca} was  extended to non-homogeneous Dirichlet data by Mascia, Porretta and Terracina in \cite{MPT}.   An $L^1$ theory for the Cauchy problem for anisotropic degenerate parabolic equations was established by Chen and Perthame \cite{CP}, based on the kinetic formulation (see \cite{PB}), and later also obtained using Kruzhkov's approach in \cite{BK,CK} (see also, \cite{FL}, \cite{KR}, \cite{LYW},  \cite{EJ} and the references therein). Decay of almost periodic solutions for general nonlinear systems of conservation laws of parabolic and hyperbolic types was first addressed in \cite{Fr0}, as an extension of the ideas put forth in \cite{CF0}. Only recently the problem of the decay of almost periodic solutions was retaken, specifically for scalar hyperbolic conservation laws, by Panov in \cite{Pv2}, where some elegant ideas were introduced to successfully extend the corresponding result in \cite{Fr0} to general bounded measurable  Besicovitch almost periodic initial functions. \ \\

Here we establish the well-posedness and decay of Besicovitch almost periodic entropy solutions of the anisotropic degenerate parabolic-hyperbolic equation \eqref{e1.1} extending the method introduced by Chen and Perthame in \cite{CP2}, which is based on the analysis of the sequence  $v^k(t,x):=u(t+k,x)$ of time translates of the entropy solution and its limits, as well as the corresponding kinetic functions and its limits. 
The extension of the method of \cite{CP2} developed in this paper consists  in upgrading the analysis framework  from the torus $\TT^d$, which is the compactification of $\R^d$ generated by the continuous periodic functions with a fixed periodic cell,  to the Bohr compact group, $\GB_d$ , which is the compactification of $\R^d$  induced by the space of Bohr almost periodic functions, $\AP(\R^d)$, according to a classical theorem of Stone (see, e.g., \cite{DS}). In the case of the hyperbolic conservation laws, the definition and well-posedness of the entropy solutions in $\GB_d$ was established in \cite{Pv2}, where it is shown the equivalence between the solutions in $\R^d$ and $\GB_d$; these facts are extended here to the context of anisotropic degenerate parabolic-hyperbolic equations.  We remark that the decay analysis carried out in \cite{Pv2}, based on an elegant idea of reducing the original problem to a problem with  a periodic initial data in a different euclidean space, does not use the formulation of the initial value problem on the Bohr compact.  Another  basic tool  used here, motivated by \cite{Pv2}, is the contraction of the $L^1$-mean distance between two entropy solutions, which was established in \cite{Pv2} in the hyperbolic case, and is easily extended here to the anisotropic degenerate parabolic-hyperbolic case. This contraction provides the compactness in the Besicovitch space equivalent to $L^1(\GB_d)$, which is the analog of the compactness in $L^1(\TT^d)$ provided by the contraction in $L^1$-distance  between periodic entropy solutions. We also introduce the kinetic formulation in $\GB_d$. We proceed to the asymptotic decay analysis totally on the Bohr compact, using both the equation and its kinetic formulation on $\GB_d$.    
After proving the compactness of the solutions in time with values in $L^1(\GB_d)$, based on the contraction of the distance in $L^1(\GB_d)$, the same compactness is obtained for the solutions of the kinetic equation satisfied by the limits of the translating sequence. We then show that only a finite number of terms in the generalized Fourier series of the limit kinetic function contribute significantly to its $L^2$-norm,  which allows us to adapt the last part of the proof in \cite{CP2}.  We remark that although the non-degeneracy \eqref{e1.4'} is formulated for all the continuum of frequencies $\k\in\R^d$, it is only used for a discrete subset of such frequencies, as is also the case in \cite{CP2}, \cite{Pv2} and \cite{Da2}.  
 \ \\

A brief description of the organization of the rest of this paper, whose main purpose is the proof of Theorem~\ref{T:1.1}, is as follows. In Section~\ref{S:2}, we start by proving a fundamental lemma establishing the contraction of the $L^1$-mean distance between any two entropy solutions of \eqref{e1.1}-\eqref{e1.2}, which extends the corresponding result in \cite{Pv2}. Then we state the existence, uniqueness, stability and monotonicity with respect to the initial data, which are by now standard, whose proofs we just outline briefly. We then establish the preservation of the space 
$\BAP(\R^d)$ and that the entropy solution $u(t,x)$ satisfies $u\in C([0,\infty);\BAP(\R^d))\cap L^\infty(\R^d)$. In Section~\ref{S:3}, we introduce the concept of entropy solution in $(0,\infty)\X\GB_d$ and translate the properties proved in the previous section in this new context. Finally, in Section~\ref{S:4}, we establish the decay of the Besicovitch almost periodic entropy solution by upgrading the method of Chen and Perthame, in \cite{CP2}, from the torus $\TT^d$ to the Bohr compact $\GB_d$.\ \\

\section{$L^1$-mean contraction, existence, uniqueness and \eqref{e1.12}}\label{S:2}

In this section we prove Theorem~\ref{T:1.1} through a number of auxiliary results and results that establish parts of its statement. 

We begin with a proposition which plays a  central role in the proof of Theorem~\ref{T:1.1}.    We will need the following  simple technical lemma of  \cite{Pv2}, to which we refer for the proof.

\begin{lemma} \label{L:2.1} Suppose that $u(x,y)\in L^\infty(\R^n\X\R^m)$, 
$$
E=\{x\in\R^n\,:\, \text{ $(x,y)$ is a Lebesgue point of $u(x,y)$ for a.e.\ $y\in\R^m$} \}.
$$
Then $E$ is a set of full measure and $x\in E$ is a common Lebesgue point of the functions $I(x)=\intl_{\R^m}u(x,y)\rho(y)\,dy$, for all $\rho\in L^1(\R^m)$.
\end{lemma}

\begin{proposition}[$L^1$-mean contraction] \label{P:1.1}  Let $u(t,x),v(t,x) \in L^\infty(\R_+^{d+1})$ be two entropy solutions of \eqref{e1.1}-\eqref{e1.2}, with initial data $u_0,v_0\in L^\infty(\R^d)$. Then for a.e.\ $0< t_0<t_1$ 
\begin{equation}\label{e2.1}
N_1(u(t_1,\cdot)-v(t_1,\cdot))\le N_1(u(t_0,\cdot)-v(t_0,\cdot)),
\end{equation}
and also for a.e.\ $t>0$, 
\begin{equation}\label{e2.2}
N_1(u(t,\cdot)-v(t,\cdot))\le N_1(u_0-v_0),
\end{equation}
\end{proposition}

\begin{proof} The proof is a slight adaptation of the one of proposition~1.3 in \cite{Pv2}. We first recall that by using the doubling of variables method of Kruzhkov \cite{Kr}, as adapted by Carrillo \cite{Ca} to the isotropic degenerate parabolic case and \cite{BK} to the anisotropic one, we obtain
\begin{equation}\label{e2.3}
|u-v|_t+\nabla\cdot \sgn(u-v)(\bff(u)-\bff(v))\le \sum_{i,j=1}^d \po_{x_i x_j}^2 \sgn(u-v)(A_{ij}(u)-A_{ij}(v))
\end{equation}
in the sense of distributions in $\R_+^{d+1}$. As usual, we define a sequence approximating the indicator function of the interval $(t_0,t_1]$ , by setting for $\nu\in\N$,
$$
\d_\nu(s)=\nu\s(\nu s),\quad \theta_\nu(t)=\int_0^t\d_\nu(s)\,ds=\int_0^{\nu t}\s(s)\,ds,
$$
where $\s\in C_c^\infty(\R)$, $\supp \s\subset [0,1]$, $\s\ge0$, $\int_{\R}\s(s)\,ds=1$.   We see that $\d_\nu(s)$ converges to the Dirac measure in the sense of distributions in 
$\R$ while $\theta_\nu(t)$ converges everywhere to the Heaviside function. For $t_1>t_0>0$, if $\chi_\nu(t)=\theta_\nu(t-t_0)-\theta_\nu(t-t_1)$, then $\chi_\nu\in C_c^\infty(\R_+)$,
$0\le\chi_\nu\le 1$, and the sequence $\chi_\nu(t)$ converges everywhere, as $\nu\to\infty$, to the indicator function of the interval $(t_0,t_1]$. Let us take $g\in C_c^\infty(\R^d)$, satisfying $0\le g\le 1$, $g(y)\equiv 1$ in the cube $\I_1$, $g(y)\equiv 0$ outside the cube $\I_k$, with $k>1$. We apply \eqref{e2.3} to the test function $\varphi=R^{-d}\chi_\nu(t)g(x/R)$, for $R>0$. We then get
\begin{multline}\label{e2.4}
\int_0^\infty\bigl(R^{-d}\intl_{\R^d}|u(t,x)-v(t,x)|g(x/R)\,dx\bigr)(\d_\nu(t-t_0)-\d_\nu(t-t_1))\,dt\\
+R^{-d-1}\iint_{\R_+^{d+1}}\sgn(u-v)(\bff(u)-\bff(v))\cdot\nabla_yg(x/R)\chi_\nu(t)\,dx\,dt\\
+R^{-d-2}\sum_{i,j=1}^d\iint_{R_+^{d+1}}\sgn(u-v)( A_{ij}(u)-A_{ij}(v))\po_{y_iy_j}^2g(x/R)\chi_\nu(t)\,dx\,dt\ge0.
\end{multline}
Define
$$
F=\{t>0\,:\, \text{$(t,x)$ is a Lebesgue point of $|u(t,x)-v(t,x)|$ for a.e.\ $x\in\R^d$} \}.
$$
As a consequence of Fubini's theorem, $F$ is a set of full Lebesgue measure and by Lemma~\ref{L:2.1} each $t\in F$ is a Lebesgue point of the functions
$$
I_R(t)=R^{-d}\int_{\R^d}|u(t,x)-v(t,x)| g(x/R)\,dx,
$$
for all $R>0$ and all $g\in C_c(\R)$. Now we assume $t_0,t_1\in F$ and take the limit as $\nu\to\infty$ in \eqref{e2.4} to get
\begin{multline}\label{e2.5}
I_R(t_1)\le I_R(t_0)+ R^{-d-1}\iintl_{(t_0,t_1)\X\R^d}\sgn(u-v)(\bff(u)-\bff(v))\cdot\nabla_y g(x/R)\,dx\,dt\\
+R^{-d-2}\sum_{i,j=1}^d\iintl_{(t_0,t_1)\X\R^d}\sgn(u-v) ( A_{ij}(u)-A_{ij}(v))\po_{y_iy_j}^2 g(x/R)\,dx\,dt.
\end{multline}
Now, we have
\begin{multline}\label{e2.6}
R^{-d-1}\bigl |\iintl_{(t_0,t_1)\X\R^d}\sgn(u-v)(\bff(u)-\bff(v))\cdot \nabla_y g(x/R)\,dx\,dt\bigr|\\
\le R^{-1} \|\bff(u)-\bff(v)\|_\infty\iintl_{(t_0,t_1)\X\R^d}|\nabla_y g(y)|\,dy\,dt\to 0,\quad \text{as $R\to\infty$}.
\end{multline}
Also, we have
\begin{multline}\label{e2.7}
R^{-d-2}\left\vert \sum_{i,j=1}^d\iint_{R_+^{d+1}}\sgn(u-v) ( A_{ij}(u)-A_{ij}(v))\po_{y_iy_j}^2g(x/R)\chi_\nu(t)\,dx\,dt\right\vert \\ 
\le CR^{-2}\|\bfA(u)-\bfA(v)\|_\infty\iintl_{(t_0,t_1)\X\R^d}|\nabla_y^2 g(y)|\,dy\,dt\to 0.
\end{multline}
 On the other hand, we have
 $$
 N_1(u(t,\cdot)-v(t,\cdot))\le \limsup_{R\to\infty} I_R(t)\le k^d N_1(u(t,\cdot)-v(t,\cdot)),
 $$
 so taking the limit as $R\to\infty$ in \eqref{e2.5}, for $t_0, t_1\in F$, $t_0<t_1$,   we get 
 $$
 N_1(u(t_1,\cdot)-v(t_1,\cdot))\le k^{d} N_1(u(t_0,\cdot)-v(t_0,\cdot)),
 $$
 and since $k>1$ is arbitrary we can make $k\to 1+$ to get the desired result. Finally, for $t_0=0$, we use \eqref{e1.10} to send $t_0\to0+$ in \eqref{e2.5} and proceed exactly as we have just done. 
  
\end{proof}

\begin{lemma}[Uniqueness]\label{L:2.2} The problem \eqref{e1.1}-\eqref{e1.2} has at most one  entropy solution.
\end{lemma}

\begin{proof}  The proof follows through standard arguments ({\em cf.}, e.g., \cite{VH}). So, let $u,v\in L^\infty(\R_+^{d+1})$ be two weak entropy solutions. As in Proposition~\ref{P:1.1}, by using the doubling of variables method of Kruzhkov \cite{Kr}, as adapted by Carrillo \cite{Ca} to the isotropic degenerate parabolic case and \cite{BK} to the anisotropic one, we obtain
\begin{multline}\label{e2.8}
\iintl_{\R_+^{d+1}}\{|u-v|\phi_t+ \sgn(u-v)(\bff(u)-\bff(v))\cdot\nabla\phi\\+ \sum_{i,j=1}^d  \sgn(u-v)(A_{ij}(u)-A_{ij}(v))\po_{x_ix_j}^2\phi\}\,dx\,dt \ge0,
\end{multline}
for all $0\le\phi\in C_c^\infty(\R_+^{d+1})$. We take $\phi(t,x)=\rho(x)\chi_\nu(t)$, where $\rho(x)=e^{-\sqrt{1+|x|^2}}$ and $\chi_\nu$ is as in the proof of Proposition~\ref{P:1.1}.
We observe that 
$$
\sum_{i=1}^d|\po_{x_i}\rho(x)|+ \sum_{i,j=1}^d|\po_{x_i x_j}^2\rho(x)|\le C\rho(x),
$$
for some constant $C>0$ depending only on $d$. Hence, making $\nu\to\infty$, we arrive at
\begin{multline*}
\intl_{\R^d}|u(t_1,x)-v(t_1,x)|\rho(x)\,dx\le \int_{\R^d}|u(t_0,x)-v(t_0,x)|\rho(x)\,dx\\+\tilde C \int_{t_0}^{t_1}\int_{\R^d}|u(s,x)-v(s,x)|\rho(x)\,dx\,dt,
\end{multline*}
for a.e. $0<t_0<t_1$, for some $\tilde C>0$ depending only on $\bff,A$ and the dimension $d$. Therefore, using Gronwall and \eqref{e1.10}, we conclude
\begin{equation}\label{e2.9}
\intl_{\R^d}|u(t,x)-v(t,x)|\rho(x)\,dx\le e^{\tilde C t}\int_{\R^d}|u_0(x)-v_0(x)|\rho(x)\,dx,
\end{equation}
which gives the desired result.
 
\end{proof}
 
 Observing that in the same way we got \eqref{e2.9} from \eqref{e2.8}, we may get
\begin{equation}\label{e2.9'}
\intl_{\R^d}(u(t,x)-v(t,x))_+\rho(x)\,dx\le e^{\tilde C t}\int_{\R^d}(u_0(x)-v_0(x))_+\rho(x)\,dx,
\end{equation}  
from
\begin{multline}\label{e2.8'}
\iintl_{\R_+^{d+1}}\{(u-v)_+\phi_t+ \sgn(u-v)_+(\bff(u)-\bff(v))\cdot\nabla\phi\\+ \sum_{i,j=1}^d  \sgn(u-v)_+(A_{ij}(u)-A_{ij}(v))\po_{x_ix_j}^2\phi\}\,dx\,dt \ge0,
\end{multline}
where $(u-v)_+=\max\{0,u-v\}$ and $\sgn(u-v)_+=H(u-v)$ where $H(s)$ is the Heaviside function. Taking $v=k$, with $k>\|u_0\|_\infty$, and then reversing the roles of $u$ and $v$, making $u=k$ and $v=u$, with $k<-\|u_0\|_\infty$, we deduce that 
\begin{equation}\label{e2.max}
|u(t,x)|\le \|u_0\|_\infty,\quad \text{for a.e.\ $(t,x)\in\R_+\X\R^d$}.
\end{equation}

\begin{lemma}[Existence]\label{L:2.3} There exists an entropy solution to the problem \eqref{e1.1}-\eqref{e1.2}.
\end{lemma}

\begin{proof} We consider the problem \eqref{e1.1}-\eqref{e1.2} with initial function 
$$
u_{0,R}(x)=u_0(x)   \chi_{{}_{{B_R}}}(x), 
$$
where $B_R=B(0,R)$ is the open ball with radius $R$ centered at the origin. By the existence theorem in \cite{CP}, which holds for initial data in $L^1(\R^d)$, we obtain an entropy solution $u_R(t,x)$ of \eqref{e1.1}-\eqref{e1.2}${}_R$. Now, using \eqref{e2.9}, we see that, for a.e.\ $t>0$,  
\begin{multline}\label{e2.9''}
\intl_{\R^d}|u_R(t,x)-u_{\tilde R}(t,x)|\rho(x)\,dx\le e^{\tilde C t}\int_{\R^d}|u_{0,R}(x)-u_{0,\tilde R}(x)|\rho(x)\,dx\longrightarrow0,\\ \; \text{as $R,\tilde R\to\infty$}.
\end{multline}
Therefore, $u_R(t,x)$ converges in $L_\loc^1((0,\infty)\X\R^d)$ to a function $u(t,x)$, which satisfies the bound in \eqref{e2.max} since it holds for all $u_R$. It is now easy to deduce from the fact that the $u_R$'s satisfy all conditions of Definition~\ref{D:1.1} that $u(t,x)$ also satisfies all those conditions. We just observe that for the verification of \eqref{e1.9} from the fact that the $u_R$'s satisfy \eqref{e1.9}, we use the uniform boundedness in $L_\loc^1(\R_+\X\R^d)$ of 
 $$
 \sum_{k=1}^d\left(\sum_{i=1}^d\po_{x_i}\b_{ik}(u_R)\right)^2
$$
and the weak lower semi-continuity of the $L^2$-norm. Also, to prove \eqref{e1.10} we first include the initial function in \eqref{e1.9}, with $u(t,x)$ replaced by $u_R(t,x)$,  tested against any function in $C_c^\infty(\R^{d+1})$,  then take the limit as $R\to\infty$ to get an entropy inequality for $u$ including the initial function. Once we get the latter, as usual,  we use a test function of the form $\z(t)\phi(x)$, with $\z'(t)=\d_\nu(t-t_0)$, for $t\ge0$, where $\d_\nu(s)$ is as in the proof of Proposition~\ref{P:1.1},  make $\nu\to\infty$, to obtain that 
$$
\lim_{t_0\to0}\int_{\R^d}\eta(u(t_0,x))\phi(x)\,dx\le \int_{\R^d}\eta(u_0(x))\phi(x)\,dx,
$$
is valid for any convex function $\eta$, which in turn implies  \eqref{e1.10}.

\end{proof}

 We recall that the space of Stepanoff almost periodic functions (with exponent $p=1$) in $\R^d$, $\SAP(\R^d)$, is defined as the completion of the trigonometric polynomials with respect to the
norm
$$
\|f\|_S :=\sup_{x\in\R^d}\int_{\I_1(x)}|f(y)|\,dy=\sup_{x\in\R^d}\int_{\I_1}|f(y+x)|\,dy,
$$
where
$$
\I_R(x):=\{y\in\R^d\,:\, |y-x|_\infty:=\max_{i=1,\cdots,d}|y_i-x_i|\le R/2\}.
$$
Another characterization of the Stepanoff almost periodic function (S-a.p., for short)  is obtained by introducing the concept of $\ve$-period of a function $f$, that is  a vector $\tau\in\R^d$ satisfying
\begin{equation}\label{e3.eper}
\|f(\cdot+\tau)-f(\cdot)\|_S\le \ve.
\end{equation}
Let $E_S\{\ve,f\}$ denote the set of such numbers. If the set $E_S\{\ve,f\}$ is relatively dense for all positive values of $\ve$, then the function $f$ is S-a.p.\ (see, e.g., \cite{B}). By the set $E_S\{\ve,f\}$ being relatively dense it is meant that there exists a length $l_\ve$, called {\em $\ve$-inclusion interval},  such that for any $x\in\R^d$, $\I_{l_\ve}(x)$ contains an element of $E_S\{\ve,f\}$.  Clearly, S-a.p.\ functions in $\R^d$ are in $\BAP(\R^d)$.

 \begin{lemma}\label{L:4.1} If $u_0$ is a trigonometric polynomial, then the entropy solution $u(t,x)$ of \eqref{e1.1}-\eqref{e1.2} is S-a.p.\ for all $t>0$, and, for any $\ve>0$, 
  $u(t,x)$ possesses an $\ve$-inclusion interval, $l_\ve(t)$, satisfying $l_\ve(t)=l_{\ve'(\ve, t)}(0)$, where $l_{\ve'}(0)$ is an $\ve'$-inclusion interval of $u_0(x)$, and $\ve'(\ve,t)=
  \ve e^{-\tilde Ct}/c_1 $, for certain $C, c_1>0$. As a consequence, if $u_0\in \BAP(\R^d)\cap L^\infty(\R^d)$, then $u(t,\cdot)\in \BAP(\R^d)$ for a.e.\ $t>0$.   
\end{lemma}

\begin{proof} Clearly, $u_0$, being a trigonometric polynomial, is S-a.p. The fact that $u(t,x)$ is S-a.p.\ for all $t>0$ follows from \eqref{e2.9}, with $v(t,x)=u(t,x+\tau)$ and $\rho(x-x_0)$ instead of $\rho(x)$, from which we deduce
\begin{eqnarray}
&&\intl_{\I_1(x_0)}|u(t,x+\tau)-u(t,x)| \,dx \label{e4.1}\\
& &\le c(t)\intl_{\I_R(x_0)}|u_0(x)-u_0(x+\tau)|\rho(x-x_0)\,dx + c(t)O\left(\frac1{R}\right) \nonumber \\
& &\le c(R,t)\sup_{x\in\R^d}\intl_{\I_1(x)}|u_0(y+\tau)-u_0(y)|\,dy+c(t)O\left(\frac1{R}\right),\nonumber
\end{eqnarray}
where $c(t)=ce^{\tilde C t}$ with, $c=e^{\sqrt{1+d}}$,  $\tilde C>0$ only depending on $\rho$,  $c(R,t)$ is a positive constant depending only on $R,t$, and $O(1/R)$ goes to zero when $R\to\infty$ uniformly with respect to $x_0$. We remark that, with $c$ just defined,  $c\rho(x-x_0)\ge 1$ on the cube $\I_1(x_0)$. More specifically,
\begin{multline*}  
\int_{\I_R(x_0)}|u_0(x)-u_0(x+\tau)|\rho(x-x_0)\,dx\le \int_{\R^d}|u_0(x)-u_0(x+\tau)|\rho(x-x_0)\,dx\\
=\sum_{y\in\Z^d}\int_{\I_1(x_0+y)}|u_0(x)-u_0(x+\tau)|\rho(x-x_0)\,dx\le c_0\sup_{y\in\R^d}\int_{\I_1(y)}|u_0(x)-u_0(x+\tau)|\,dx,
\end{multline*}
where
$$
c_0=\sum_{y\in\Z^d}\max_{\I_1(x_0+y)}\rho(x-x_0)=\sum_{y\in\Z^d}\max_{\I_1(y)}\rho(x)<+\infty.
$$
Therefore, we can choose $C(R,t)=c_0c(t)=c_1e^{\tilde C t}$, $c_1=cc_0$. In particular, $C(R,t)=C(t)$ does not depend on $R$.  
So, choosing $R$ large enough so that $c(t)O(1/R)\le \ve/2$ and then taking any $\tau\in E_S\{\ve/(2C(t)), u_0\}$, we get that $\tau \in E_S\{\ve, u(t,\cdot)\}$, and so $u(t,\cdot)$ is S-a.p. By the above calculation, we get the estimate $l_\ve(t)=l_{\ve'(\ve, t)}(0)$, with   $\ve'(\ve,t)=\ve e^{-\tilde Ct}/c_1$. 

As for the final assertion, given $u_0$ satisfying \eqref{e1.5}, we approximate $u_0$ by trigonometric polynomials, say, using Bochner-F\'ejer's polynomials (see \cite{B}). Then, we use the Proposition~\ref{P:1.1} to obtain that the solutions corresponding to the approximating trigonometric polynomials converge in the 
$N_1$-seminorm uniformly in $t$ to the entropy solution associated to $u_0$, and so we have $u(t,\cdot)\in \BAP(\R^d)$ for a.e.\ $t>0$.  

\end{proof}

We prove now the continuity of the (weak) entropy solution of \eqref{e1.1}-\eqref{e1.2} as a function from $[0,\infty)$ to $\BAP(\R^d)$.

\begin{lemma}\label{L:4.2} Let $u$ be the  entropy solution of  \eqref{e1.1}-\eqref{e1.2}. Then $u\in C([0,\infty); L^1(\I_R))$, for any $R>0$. Moreover,  if $u_0$  satisfies \eqref{e1.5}, then $u\in C([0,\infty), \BAP(\R^d))$.
\end{lemma}

\begin{proof}  We first show that $u\in C([0,\infty); L^1(\I_R))$, for any $R>0$. By the uniqueness (see Lemma~\ref{L:2.2}), we may assume that $u$ is obtained as the limit of the solutions of the parabolic approximate problems with a vanishing viscosity $\ve>0$,
\begin{align}
& \po_tu+\nabla\cdot \bff(u)=\nabla^2: \bfA(u)+\ve \Delta u,\qquad x\in\R^d,\quad t>0,\label{e1.1ve}\\
&u(0,x)=u_0, \quad x\in\R^d.\label{e1.2ve}
\end{align}
Using the analog of Lemma~\ref{L:2.2} for problem \eqref{e1.1ve}-\eqref{e1.2ve} we obtain a uniform in $\ve$ and $t$, for $0\le t\le T$,  modulus of continuity 
$\om_r^x(\s)$  for 
$$
J_r(u^\ve(t,x),\Delta x):=\int_{\I_r}|u^\ve(t,x+\Delta x)-u^\ve(t,x)|\,dx\le \om_r^x(|\Delta x|),
$$
 where we denote by $u^\ve(t,x)$ the solution of  \eqref{e1.1ve}--\eqref{e1.2ve}.  Then we use lemma~5 of \cite{Kr} to obtain a uniform in $\ve$ and $t$, for $0\le t\le T$,  modulus of continuity in $L_\loc^1(\R^d)$ in the $t$ variable of the form
 $$
 I_r(u^\ve(t,x),\Delta t):=\int_{\I_r}|u^\ve(t+\Delta t,x)-u^\ve(t,x)|\,dx\le \text{const.\ }\min_{0<h\le \rho}\left[ h+ \om_r^x(h)+\frac{\Delta t}{h^2}\right],
 $$
 with $0<2\rho\le r$. These estimates imply the compactness of the sequence $u^\ve(t,x)$ in $L_\loc^1((0,\infty)\X\R^d)$ and the limit $u(t,x)$ also satisfies both estimates.
 In particular, $u\in C([0,\infty); L^1(\I_R))$, for any $R>0$.  

To prove that $u\in C([0,\infty), \BAP(\R^d))$, let us first consider the case where $u_0$ is a trigonometric polynomial. By Lemma~\ref{L:4.1}, the corresponding entropy solution of \eqref{e1.1}-\eqref{e1.2} is S-a.p.\ for all $t>0$. Also, given any $T>0$, and $\ve>0$, we can get $l_\ve$ sufficiently large which is an $\ve$-inclusion interval for $u(t,\cdot)$ for all $t\in [0,T]$.  Let us partition $\R^d$ through the net of cubes $I'=kl_\ve+[0,l_\ve]^d$, $k\in\Z^d$, with edges of length $l_\ve$ parallel to the axes. For each such cube $I'$ there exists an $\ve$-almost period $\tau_{I'}$ such that $I'-\tau_{I'}\subset [0,2 l_\ve]^d$. Hence, given $t,s\in [0,T]$, we may  assume for simplicity that the $\tau_{I'}$ are  common $\ve$-almost periods for both $u(t,\cdot)$ and $u(s,\cdot)$. We then have
\begin{align}
\Me(|u(t,\cdot)-u(s,&\cdot)|)=\lim_{N\to\infty} \frac{1}{(2Nl_\ve)^d}\int_{[-Nl_\ve,Nl_\ve]^d}|u(t,x)-u(s,x)|\,dx\label{e4.2}\\
&\le \lim_{N\to\infty} \frac{1}{(2Nl_\ve)^d}\sum\limits_{I'\subset[-Nl_\ve,Nl_\ve]^d} \int_{I'}|u(t,x-\tau_{I'})-u(s,x-\tau_{I'})|\,dx\nonumber\\
&+\lim_{N\to\infty} \frac{1}{(2Nl_\ve)^d}\sum\limits_{I'\subset[-Nl_\ve,Nl_\ve]^d} \int_{I'}|u(t,x)-u(t,x-\tau_{I'})|\,dx\nonumber\\
&+\lim_{N\to\infty} \frac{1}{(2Nl_\ve)^d}\sum\limits_{I'\subset[-Nl_\ve,Nl_\ve]^d} \int_{I'}|u(s,x)-u(s,x-\tau_{I'})|\,dx\nonumber\\
&\le\frac{1}{l_\ve^d}\int_{[0,2l_\ve]^d}|u(t,x)-u(s,x)|\,dx+2^{d+1}\ve.\nonumber
\end{align}
The above inequality holds for any $t,s\in[0,T]$ and $\ve>0$. Since, as we have just shown, $u\in C([0,\infty), L_{\loc}^1(\R^d))$,  we see that for $t$ and $s$ close enough, the right-hand side of \eqref{e4.2} is $\le (1+2^{d+1})\ve$, which proves that $u\in C([0,\infty);\BAP(\R^d))$ in the case where $u_0$ is a trigonometric polynomial.    

Now, for general initial data satisfying \eqref{e1.5}, we use Proposition~\ref{P:1.1} and approximate $u_0$ by trigonometric polynomials $u_{0k}$, e.g., using Bochner-F\'ejer's polynomials (see \cite{B}),  and observe that, for each $t>0$,  the corresponding solutions $u_k(t,\cdot)$ converge to the entropy solution $u(t,\cdot)$ corresponding to $u_0$, in $\BAP(\R^d)$, uniformly for $t>0$. Therefore, we again have $u\in C([0,\infty),\BAP(\R^d))$, proving the lemma.    

\end{proof}

As a consequence of the fact that $u\in C([0,\infty);\BAP(\R^d))$ we have the following.

\begin{lemma}\label{L:4.3} The set $\Lambda_u=\{\l \in \R^d\,:\, \Me(e^{-2\pi i \l\cdot (\cdot)}u(t,\cdot))\ne0,\; \text{for some $t\ge0$} \,\}$ is at most countable. 
\end{lemma}

\begin{proof} For any $0\le t\in \Q$, we have that $\Sp(u(t,\cdot))$ is at most countable, so the set $L=\{\l\in \Sp(u(t,\cdot))\,:\, 0\le t\in \Q\}$ is at most countable. Now, for $0\le t\notin \Q$ and $\l\in \Sp(u(t,\cdot))$, if $\l\notin L$, then $\Me(e^{-2\pi i \l\cdot (\cdot)}u(t,\cdot))=0$, for all $0\le t\in \Q$. Therefore,  since $u\in C([0,\infty);\BAP(\R^d))$, it follows that $\l\notin\Lambda_u$, that is $\Lambda_u\subset L$ and so it is at most countable. 
\end{proof}

\section{The problem on the Bohr compact $\GB_d$.}\label{S:3}

For $1\le p<\infty$, let $\BAP^p(\R^d)$ denote the space of the $L^p$-Besicovitch almost periodic functions, which can be defined as the completion of the space of trigonometric polynomials, i.e., finite sums $\sum_{\l}a_\l e^{2\pi i\l\cdot x}$  under the semi-norm
$$
N_p(g):= \limsup_{R\to\infty}\left( \frac{1}{R^d}\int_{\I_R}|g(x)|^p\,dx\right)^{1/p}.
$$
In particular, $\BAP^1(\R^d)=\BAP(\R^d)$. We denote by $[|f|]_p:=\Me(|f|^p)^{1/p}=\left(\medint |f|^p\,dx\right)^{1/p}$ the norm in $\BAP^p(\R^d)$ obtained from $N_p$. 

Let $\GB_d$ denote the Bohr compact group, which is the compactification of $\R^d$ provided by $\AP(\R^d)$ through a classical theorem of Stone, such that $\AP(\R^d)$ is isometrically isomorphic to $C(\GB_d)$  (see \cite{DS}, also \cite{AF}).   $\GB_d$ is endowed with the (probability) measure induced by the mean value functional over $\AP(\R^d)$, which coincides with the normalized Haar measure inherent to its topological group structure,  and henceforth will be denoted by $\mm$. 
It then follows  that  $\BAP^p(\R^d)$ is isometrically isomorphic to $L^p(\GB_d;\mm)$, $1\le p<\infty$.   Indeed, if $f\in\BAP^p(\R^d)$, then $f$ may be approximated in the norm of $\BAP^p(\R^d)$ by a sequence in $\AP(\R^d)$, each of whose functions may be viewed as an element of $C(\GB_d)$ and together form a Cauchy sequence in $L^p(\GB_d)$, since, by definition, the $\BAP^p$-norm of functions in $\AP(\R^d)$ is equal to the norm of the associated functions in  $L^p(\GB_d)$. In this way, we obtain an element $\hat f\in L^p(\GB_d)$ associated with $f\in\BAP(\R^d)$. We can easily reverse the arguments and conclude that given any function in  $L^p(\GB_d)$ we may associate to it a unique element of $\BAP(\R^d)$, $1\le p<\infty$.  The isometric isomorphism $f\mapsto \hat f$ between $\AP(\R^d)$ and $C(\GB_d)$ and
its extensions between $\BAP^p(\R^d)$ and $L^p(\GB_d)$, $1\le p<\infty$,  is sometimes referred to as Gelfand transform.   

We can also define $\BAP^\infty(\R^d)$ in the following way
$$
\BAP^\infty(\R^d):=\{f\in \bigcap_{1\le p<\infty}\BAP^p(\R^d)\,:\, \sup_{p\in[1,\infty)}[|f|]_p<\infty\,\},
$$
and define for $f\in\BAP^\infty(\R^d)$
$$
[|f|]_\infty:=\sup_{p\in[1,\infty)}[|f|]_p.
$$
In this way, the association $f\mapsto \hat f$ is also an isometric isomorphism between $\BAP^\infty(\R^d)$ and $L^\infty(\GB_d)$. 

Given $f\in \BAP^\infty(\R^d)$ with $M=[|f|]_\infty$, defining
$$
f_M(x):= \begin{cases} -M, & \text{if $f(x)\le -M$} \\ f(x), & \text{if $|f(x)|\le M$}\\ M, &\text{if $f(x)\ge M$} \end{cases}, 
$$
we have the identity  $f_M\equiv f$ in $\BAP^p(\R^d)$, for $1\le p\le\infty$. So, any element in $\BAP^\infty(\R^d)$ has a representative in $\BAP^1(\R^d)\cap L^\infty(\R^d)$. The converse is trivially true, that is, given any $f\in\BAP^1(\R^d)\cap L^\infty(\R^d)$, we have that $f\in\BAP^\infty(\R^d)$ and it is easy to see that 
$$
[|f|]_\infty\le \|f\|_\infty.
$$

\medskip

Concerning the mean value, it is well  known that if a function $g\in L_\loc^1(\R^d)$ possesses a mean value $\Me(g)$,  then, for all $\phi\in C_c(\R^d)$, 
$$ 
 \lim_{R\to\infty}\int_{\R^d} g(R x)\phi(x)\,dx= \Me(g)\int_{\R^d} \phi(x)\,dx,
 $$
 and, reciprocally, the latter also serves as a definition for the mean value. This relation can also be written, by a trivial change of coordinates, in the form
 \begin{equation}\label{e3.1}
 \lim_{R\to\infty} R^{-d}\int_{\R^d} g(x) \phi(\frac{x}{R})\,dx= \Me(g)\int_{\R^d} \phi(x)\,dx, \quad \forall \phi\in C_c(\R^d).
 \end{equation}
 
  \medskip
  Concerning the structure of topological (commutative) group of which $\GB_d$ is endowed, another important consequence is the existence of an approximate identity, that is, a (generalized) sequence $\{\rho_\a\,:\, \a\in J\}\subset \AP(\R^d)$, where $J$ is the partially ordered  set of neighborhoods of $0$ in $\GB_d$,      
satisfying  $\rho_\a\ge0$, $\Me(\rho_\a)=1$, for all $\a$, $\supp \rho_\a \subset \bar\a$, where $\bar \a$ is the closure of the neighborhood $\a$, with the property that $\rho_\a* f\to f$, as $\a\to \{0\}$, for all $f\in\AP(\R^d)$, and we may assume  $\rho_\a(-x)=\rho_\a(x)$ (see \cite{Loo}).   
Here, ``$*$'' is the convolution operation naturally defined in $\AP(\R^d)$ by
\begin{equation*}
f*g(x)=\Medint_{\R^d}f(x-y)g(y)\,dy=\Medint_{\R^d}f(y)g(x-y)\,dy,
\end{equation*}
or, viewed as an operation in $C(\GB_d)$,
$$
\hat f*\hat g(\om)=\int_{\GB_d}\hat f(\om-\z)\hat g(\z)\,d\mm(\z)=\int_{\GB_d}\hat f(\z)\hat g(\om-\z)\,d\mm(\z).
$$
In what follows we will frequently identify functions in $\BAP(\R^d)$ with their Gelfand transforms in $L^1(\GB_d)$ omitting the `` $\hat{}$ ''.

The compact $\GB_d$ is a non-separable topological space and so the set of neighborhoods of 0 has no countable basis. On the other hand, it is often preferable to work with sequences than with generalized sequences, but the approximate identity in $\GB_d$ is in general a generalized sequence. However, it is possible to introduces coarser topologies on $\GB_d$ which are not Hausdorff but whose quotient $\GB_d/\sim$, with $\om\sim\z$ iff $\om$ and $\z$ cannot each belong to a neighborhood that does not contain
 the other, is a compact topological group and so also endowed with an approximation of the identity. In what follows we will introduce such topologies by means of closed subalgebras  containing the identity of the algebra generated by $\{e^{2\pi i\l\cdot x}\,:\, \l\in\R^d\}$. The point is that if all functions on $\GB_d$ that you are going to deal with are Borel functions on $\GB_d$ with respect to topology generated by such subalgebra, then you may use as approximate identity an approximate identity for $\GB_d/\sim$, which is the same as restricting the neighborhoods of 0 to a countable basis of neighborhoods for the topology generated by the corresponding subalgebra.          

Given a family $\F$ of functions in  $\BAP(\R^d)$, we denote by $\A_\F$ the closure in the $\sup$-norm of the algebra, over the complex numbers, generated by $\{1,\ e^{2\pi i \l\cdot x}\,:\,  \l\in\Sp(v),\; v\in\F\}$ and by $\Gr(\F)$ the smallest additive group generated by $\Lambda_\F:=\{\l \in \Sp(v)\,:\, v\in\F\}$.   In the particular case where $\F=\{u(t,\cdot)\,:\,t\ge0\}$, with $u(t,\cdot)\in \BAP(\R^d)$, for all $t\ge0$, we use the simplified notation $\Lambda_u, \Gr(u), \A_u$ instead of $\Lambda_\F,\Gr(\F), \A_\F$, respectively. Similarly, 
when $\F:=\{ u(t,\cdot),\; v(t,\cdot)\,:\, t\ge0\}$, we use the simplified notation $\Lambda_{u,v}, \Gr(u,v), \A_{u,v}$ instead of $\Lambda_\F,\Gr(\F), \A_\F$, respectively. 

If $u(t,\cdot)$ and $v(t,\cdot)$ are entropy solutions of \eqref{e1.1}-\eqref{e1.2}, satisfying \eqref{e1.5}, Lemma~\ref{L:4.3} implies that $\Lambda_u$, $\Lambda_v$ and $\Lambda_{u,v}$ are countable sets, and so
$\A_u$, $\A_v$ and $\A_{u,v}$ generate separable topologies in $\GB_d$, that is, topologies endowed with a countable basis of neighborhoods of $0$. Moreover, if $u(t,\cdot)$ and $v(t,\cdot)$ are two entropy solutions of \eqref{e1.1}-\eqref{e1.2}, satisfying \eqref{e1.5}, and $h\in C([-M,M])$, $H\in C([-M,M]\X[-M,M])$,  $M=\|u_0\|_\infty$, then
$$
\Sp(h(u(t,\cdot)))\subset \Gr(u),\quad \Sp(H(u(t,\cdot),v(t,\cdot)))\subset \Gr(u,v), \quad \forall t\ge0,
$$
since $h(u(t,\cdot))$ and $H(u(t,\cdot),v(t,\cdot))$ may be approximated in $\BAP(\R^d)$ by functions in $\A_u$ and $\A_{u,v}$, respectively, since, as already seen, $u(t,\cdot), v(t,\cdot)\in\BAP^\infty(\R^d)$.   

In particular, if all the $\BAP$ functions we are considering are of the form $h(u(t,\cdot))$ or $H(u(t,\cdot),v(t,\cdot))$, where $u(t,\cdot)$ and $v(t,\cdot)$ are entropy solutions of \eqref{e1.1}-\eqref{e1.2}, satisfying \eqref{e1.5}, we   may then restrict the indices of  the approximate identity  $\rho_\a$ so that $\a$  runs through the countable basis of neighborhoods of $0$ in $\GB_d$ belonging to the topology generated by  $\A_u$, or $\A_{u,v}$, as the case may be. We remark in passing that the topology in $\GB_d$ generated by the closed algebra $\A_u$ is not Hausdorff, but the compact space associated with it, via Stone theorem, may be obtained by taking the quotient of $\GB_d$ by the equivalence relation 
$\sim$, where $\om\sim\z$ means $\om$ and $\zeta$ cannot be separated by $\A_u$, that is, $g(\om)=g(\z)$ for all $g\in\A_u$. 
  
 In what follows we will be dealing with functions of these types so we will use the approximate identity $\rho_\a$ assuming  $\a\in\N$; in particular,  $\rho_\a*g\to g$,  as $\a\to\infty$, in the $\sup$-norm, for all $g\in \A_u$, or $\A_v$, or $\A_{u,v}$, according to the case. 

More generally, if $u(t,x)$ is the entropy solution of \eqref{e1.1}-\eqref{e1.2}, satisfying \eqref{e1.5}, and $g:[-M,M] \to\R$, $M=\|u_0\|_\infty$,  is a bounded Borel function, we have that $\Sp(g(u(t,\cdot)))\subset \Gr(u)$,
 for all $t\ge0$, where the composition $g(u(t,x))$ is defined by $\widehat{g(u)}=g(\hat u)$, and $\hat{}$ denotes the Gelfand transform.  Indeed, this is easily verified for $g\in C([-M,M])$. On the other hand, the class of bounded Borel functions $g$ satisfying  $\Sp(g(u(t,\cdot)))\subset \Gr(u)$, for all $t\ge0$, is closed under everywhere convergence, by dominated convergence.

The following lemma establishes an important fact about the entropy solution of \eqref{e1.1}-\eqref{e1.2}.

\begin{lemma}\label{L:3.1} Let $u(t,x)$ be the entropy solution of \eqref{e1.1}-\eqref{e1.2}. Then, the set $\{u(\cdot,t)\,:\, t>0\}$ is relatively compact in $\BAP^2(\R^d)$.  
\end{lemma}

\begin{proof} We first observe that, by Proposition~\ref{P:1.1}, we have, for any $h\in\R^d$,
\begin{equation}\label{e4.24}
\Me(|u(t,\cdot+h)-u(t,\cdot)|^2)\le 2\|u\|_\infty\Me(|u_0(\cdot+h)-u_0(\cdot)|).
\end{equation}
 Identifying functions with their Gelfand transform, this may also be written as
\begin{equation}\label{e4.25}
\int_{\GB_d}|u(t,\z+h)-u(t,\z)|^2\,d\mm(\z)\le   2\|u\|_\infty\int_{\GB_d}|u_0(\z+h)-u_0(\z)|\,d\mm(\z).
\end{equation}
Let $\rho_\a$ be an approximation of the identity in $\AP(\R^d)$ and, as just discussed in the paragraph before the statement of the lemma, since we are only interested in applying $\rho_\a$ to functions which belong to the closure in the $\BAP$-norm of  the space $A_u$, we may assume that $\a$ runs through $\N$ or a countable decreasing family of neighborhoods of 0.  Now,  for each $\a$,  the set of almost periodic functions $\{ g_t^\a(x):=[\rho_\a*u(t, \cdot)](x)\,:\, t>0\}$ is an equicontinuous family of almost periodic functions. Indeed, again by Proposition~\ref{P:1.1}, we have
\begin{multline*}
|g_t^\a(x+h)-g_t^\a(x)|\le \|\rho_\a\|_\infty\Me(|u(t,\cdot+h)-u(t,\cdot)|)\\ \le\|\rho_\a\|_\infty\Me(|u_0(\cdot+h)-u_0(\cdot)|).
\end{multline*}

 Moreover, the set of almost periodic functions $\{ g_t^\a\}_{t>0}$ has the following property: given any $\ve>0$, there exist $\l_1,\cdots,\l_N\in\R^d$, such that, for any $\l\in\R^d$, there is $\l_j$, with $j\in\{1,\cdots,N\}$, with $\|g_t^\a(\cdot+\l)-g_t^\a(\cdot+\l_j)\|_{\infty}<\ve$, for all $t>0$. Indeed,  since
$$
g_t^\a(x+\l)=\Me(\rho_\a(\cdot+\l)u(t,x-\cdot)) 
$$
this follows from the almost periodicity of $\rho_\a$, which guarantees that, given $\ve'>0$, we have $\l_1,\cdots,\l_N\in\R^d$ so that, for any $\l\in\R^d$, there exists $\l_j$, $j\in\{1,\cdots,N\}$, with $\|\rho(\cdot+\l)-\rho(\cdot+\l_j)\|_{\infty}<\ve'$. Therefore,
$$
|g_t^\a(x+\l)-g_t^\a(x+\l_j)|\le \|\rho(\cdot+\l)-\rho(\cdot+\l_j)\|_\infty\Me(|u(t,\cdot)|)< \Me(|u_0|)\ve',
$$
where we use Proposition~\ref{P:1.1}. Hence, we can invoke a well known criterion by Lyusternik (see, e.g., \cite{LZ}, \cite{AF}) to conclude that, for each fixed $\a$, the set $\{g_t^\a\}_{t>0}$ is relatively compact in $\AP(\R^d)$. Now, we have
\begin{align}\label{e4.26}
\Me(|\rho_\a*u(t,\cdot)-u(t,\cdot)|^2)&\le2\|u\|_\infty \int_{\GB_d} \left |\int_{\GB_d}\rho_\a(\z)(u(t,\om-\z)-u(t,\om))\,d\mm(\z)\right |\,d\mm(\om)\\
      &\le 2\|u\|_\infty\int_{\GB_d} \int_{\GB_d}\rho_\a(\z)|u(t,\om-\z)-u(t,\om)|\,d\mm(\z)\,d\mm(\om)\nonumber\\
      &\le 2\|u\|_\infty\sup_{\z\in\supp \rho_\a}\int_{\GB_d}|u(t,\om-\z)-u(t,\om)|\,d\mm(\om)\nonumber\\
      &\le 2\|u\|_\infty \sup_{\z\in\supp \rho_\a}\int_{\GB_d}|u_0(\om-\z)-u_0(\om)|\,d\mm(\om),\nonumber
      \end{align}
where we use Fubini and, once more,  Proposition~\ref{P:1.1}. Concerning the latter,  it implies 
$$
\int_{\GB_d}|u(t,\om-\z)-u(t,\om)|\,d\mm(\om)\le \int_{\GB_d}|u_0(\om-\z)-u_0(\om)|\,d\mm(\om)
$$
first for all $\z\in\R^d$, and then for all $\z\in \GB_d$, by the continuity in $\z\in\GB_d$ of the translations  $T_\z v(\cdot)=v(\cdot-\z)$ on  $L^1(\GB_d)$, recalling that $\GB_d$ is a topological group and $\,d\mm$ is its Haar measure.  

Now, since $\{\supp\rho_\a\}$ is a countable decreasing family of neighborhoods of $0$ converging to $\{0\}$, we deduce that the set $\{g_t^\a\}_{t>0}$ is as close as we wish to $\{u(t,\cdot)\}_{t>0}$, in $\BAP^2(\R^d)$. Hence, the relative compactness of $\{g_t^\a\}$ in $\AP(\R^d)$, for arbitrary $\a$,  implies the relative compactness of $\{u(t,\cdot)\}_{t>0}$ in $\BAP^2(\R^d)$.  Indeed, let $u(t_k,\cdot)$ be a bounded sequence in 
$\BAP^2(\R^d)$. For each $\a$, which for simplicity we may assume to run through $\N$,  $\{g_{t_k}^\a\}$  is relatively compact and so it possesses a converging subsequence in $\AP(\R^d)$. So, for $\a=1$, there is a subsequence $t_{1k}$ such that $g_{t_{1k}}^1$ converges in the $\sup$-norm as $k\to\infty$. Similarly, for $\a=2$, we may extract a subsequence $\{t_{2k}\}\subset\{t_{1k}\}$ such that  $g_{t_{2k}}^1$ converges in the $\sup$-norm, and so on. We claim that $u(t_{kk},\cdot)$ is a convergent sequence in $\BAP^2(\R^d)$.  In fact, given $\ve>0$, for $\a$ sufficiently large, $[| g_{t_{\a k}}^\a-u(t_{\a k},\cdot)|]_2<\ve/3$, by \eqref{e4.26}, for all $k\in\N$.
Now, since $g_{t_{\a k}}^\a$ converges, for $k, l>N_0>\a$, for some $N_0$ sufficiently large, $\|g_{t_{\a k}}^\a-g_{t_{\a l}}^\a\|_\infty <\ve/3$. Hence, we have, for $k, l>N_0$,     
\begin{eqnarray*}
&&[|u(t_{kk},\cdot)-u(t_{ll},\cdot)|]_2\\
&& \le [|u(t_{kk},\cdot)-g_{t_{k k}}^k(\cdot)|]_2+[|g_{t_{kk}}^k(\cdot)-g_{t_{ll}}^l(\cdot)|]_2+[|g_{t_{ll}}^l(\cdot)-u(t_{ll},\cdot)|]_2\\
&&<\frac{\ve}3+\frac{\ve}3+\frac{\ve}3=\ve.
\end{eqnarray*}

\end{proof}

Now, the compactness of $\{u(t,\cdot)\,:\, t>0\}$ in $\BAP^2(\R^d)$ implies, in particular, the following.

\begin{lemma}\label{L:3.2} If we write, for any $t\ge0$,  
\begin{equation}\label{e4.27}
u(t,x)=\sum_{\l\in\Lambda_u}a_\l(t)e^{2\pi i\l\cdot x},
\end{equation}
with equality in the sense of $\BAP^2(\R^d)$,  then, given any $\ve>0$, there exists a finite set $F_\ve\subset \Lambda_u$, such that  
\begin{equation}\label{e4.28}
\sum_{\l\in\Lambda_u\setminus F_\ve}|a_\l(t)|^2 <\ve,\quad\text{for all $t\ge0$}.
\end{equation}
\end{lemma}

\begin{proof}  By compactness,  given any $\ve>0$, we may find $g_1:=u(t_1,\cdot),\cdots,g_m=u(t_m,\cdot)$, such that for any $t\ge0$,  $\Me(|u(t,\cdot)-g_\nu|^2)<\ve/4$, for some $\nu\in\{1,\cdots,m\}$.  We observe that, given any $\ve>0$,   we may find a finite set $F_\ve$ such that 
$$
\sum_{\l\in\Lambda_u\setminus F_\ve}|a_\l(t_\nu)|^2<\ve/4,\quad \nu=1,\cdots,m.  
$$

Therefore,
$$
\sum_{\l\in\Lambda_u\setminus F_\ve}|a_\l(t)|^2\le 2\Me(|u(t,\cdot)-g_\nu(\cdot)|^2)+2\sum_{\l\in\Lambda_u\setminus F_\ve}|a_\l(t_\nu)|^2<\frac{\ve}2+\frac{\ve}2=\ve.
$$ 

\end{proof}

We are going to define entropy solution of the problem corresponding to \eqref{e1.1}-\eqref{e1.2} in $\GB_d$. A similar procedure was carried out in \cite{Pv2} for the case of the hyperbolic conservation laws. We point out that in \cite{Pv2} the role of the approximate identity is played by the Bochner-Fej\'er kernels, whose explicit formula in the multidimensional case is given in \cite{Pv2}, based on that for the one-dimensional case given in \cite{B}. Here we use the approximate identity $\rho_\a$, with $\a$ running through the countable set of neighborhoods of 0 in $\GB_d$ forming a basis for the topology generated by $\A_u$, because they are supported in the corresponding neighborhoods and this fact simplifies some arguments.    

\begin{lemma} \label{L:3.3} Let $v\in\BAP(\R^d)$,  and $\rho_\a$,  be an approximation of the identity with $\a$ running through a countable basis of neighborhoods of 0 in the topology generated by $\A_v$. Then
\begin{equation}\label{e4.29}
\lim_{\a\to\infty} \Medint_{\R^d}\left(\Medint_{\R^d}|v(x)-v(y)|\rho_\a(x-y)\,dy\right)\,dx=0.
\end{equation}
Equivalently, identifying $v$ and $\rho_\a$ with their Gelfand transforms, this may be written as  
\begin{equation}\label{e4.30}
\lim_{\a\to\infty} \int_{\GB_d}\left(\int_{\GB_d}|v(\om)-v(\z)|\rho_\a(\om-\z)\,d\mm(\om)\right)\,d\mm(\z)=0.
\end{equation}
\end{lemma}

\begin{proof} Since the Gelfand transform is an isomorphism between $\AP(\R^d)$ and $C(\GB_d)$ and between $\BAP^p(\R^d)$ and $L^p(\GB_d)$, $1\le p\le \infty$, where the Haar measure $\mm$ in $\GB_d$ is induced by the mean-value over $\R^d$, it suffices to prove \eqref{e4.30}. Now, any function $v\in\BAP(\R^d)$ may be approximated in the
$\BAP$-norm by functions $v_r\in\A_v$, such as its Bochner-Fej\'er trigonometrical polynomials (see, {\em e.g.}, \cite{B}), which is the same to say that $v_r\in\A_v$ and $v_r\to v$
in $L^1(\GB_d)$.  Using the triangular inequality $||v(\om)-v(\z)|-|v_r(\om)-v_r(\z)||\le |v(\om)-v_r(\om)|+|v(\z)-v_r(\z)|$ we reduce the problem to verifying \eqref{e4.30} when $v\in\AP(\R^d)$.
But, for $v\in\AP(\R^d)$, $v$ is  uniformly continuous in $\GB_d$, in the sense of the natural uniformities in topological groups (see \cite{Loo}), and so, given $\ve>0$,  $|v(\om)-v(\z)|<\ve$, for $\z\in \om+\V_\a$, where $\V_\a$ is a neighborhood of $0$ sufficiently small, for all $\om\in\GB_d$. So,
$$
\int_{\GB_d}|v(\om)-v(\z)|\rho_\a(\om-\z)\,d\mm(\z) <\ve,
$$
for $\a$ sufficiently large, uniformly in $x\in\GB_d$, which proves \eqref{e4.30} in case $v\in \AP(\R^d)$, and, as already shown, this suffices to conclude the proof in the general case where $v\in\BAP(\R^d)$. 

\end{proof}

For a function $f\in C((0,\infty)\X\GB_d)$ we define the partial derivatives $f_t,f_{x_i}$ in the usual way 
$$
f_t(t,x)=\lim_{h\to0}\frac{f(t+h,\om)-f(t,\om)}h,\quad f_{x_i}(t,\om)=\lim_{h\to0}\frac{f(t,\om+h e_i)-f(t,\om)}h,
$$
whenever these limits exist, for some $(t,\om)\in(0,\infty)\X\GB_d$, where $e_i$ is the $i$-the element of the canonical basis.  When these derivatives exist at all $(0,\infty)\X\GB_d$
we may define derivatives $D^\b f$, for any multi-index $\b=(k_0,k_1,\cdots,k_d)$ in an inductive way, that is, assuming that $D^\b f$ exists for all multi-index with $|\b|<k$, we define $D^\g f$, for $|\g|=k+1$, writing $D^\g f=D^{\b_0}D^\b f$, with $|\b_0|=1$, $|\b|=k$, and defining $D^\b f$ by the limits above, whenever they exist,  with $D^\b f$ in the place
 of $f$.     

We denote by $C^k((0,\infty)\X\GB_d)$ the functions $f\in C((0,\infty)\X\GB_d)$ such that $D^\b f\in C((0,\infty)\X\GB_d)$, for all multi-indices $\b=(k_0,k_1,\cdots,k_d)$ with $|\b|\le k$,  and by $C_c^k((0,\infty)\X\GB_d)$ the functions in $C^k((0,\infty)\X\GB_d)$ with support compact in $(0,\infty)\X\GB_d$. We denote $f\in C_c^\infty((0,\infty)\X\GB_d)$ if
$f\in C_c^k((0,\infty)\X\GB_d)$ for all $k\in\N$. 

Given $f\in L_\loc^1((0,\infty)\X\GB_d)$ we say that $f^\b=D^\b f$ in the sense of distributions in $(0,\infty)\X\GB_d$, if 
\begin{equation}\label{e4.30'}
\int_{\R_+}\int_{\GB_d} f D^\d\phi\, d\mm(\om)\,dt=(-1)^{|\b|}\int_{\R_+}\int_{\GB_d} f^\b \phi \, d\mm(\om)\,dt,
\end{equation}
for all $\phi\in C_c^\infty((0,\infty)\X\GB_d)$. Equation \eqref{e4.30'} is coherent with the usual integration by parts for functions in $C_c^\infty((0,\infty)\X\GB_d)$. Indeed, it suffices to check the case of a single space derivative. For $\phi,\psi\in C_c^\infty((0,\infty)\X\GB_d)$, we have
\begin{multline*}
\int_{\R_+}\int_{\GB_d} \psi \po_{x_i}\phi\, d\mm(\om)\,dt=\lim_{R\to\infty}\frac1{R^d}\int_0^\infty \int_{\I_R}\psi\po_{x_i}\phi\,dx\,dt\\
=-\lim_{R\to\infty}\frac1{R^d}\int_0^\infty \int_{\I_R}\po_{x_i}\psi\,\phi\,dx\,dt +\lim_{R\to\infty}\frac1{R^d}\int_0^\infty \int_{\po\I_R}\psi\phi\nu_i\,dx\,dt\\
=-\lim_{R\to\infty}\frac1{R^d}\int_0^\infty \int_{\I_R}\po_{x_i}\psi\,\phi\,dx\,dt\\
=-\int_{\R_+} \int_{\GB_d}\po_{x_i}\psi\,\phi\,d\mm(\om)\,dt,  
\end{multline*}
where $\nu_i$ is the $i$-th component of the unity outer normal to $\po\I_R$.

We are going to define entropy solution for the problem
\begin{align}
& \po_tu+\nabla\cdot \bff(u)=\nabla^2: \bfA(u),\qquad x\in\GB_d,\quad t>0,\label{e4.31}\\
&u(0,\cdot)=u_0\in L^\infty(\GB_d). \label{e4.32}
\end{align}

\begin{definition}\label{D:4.1}  An entropy solution for \eqref{e4.31},\eqref{e4.32},  is a function $u(t,\om)\in L^\infty((0,\infty)\X\GB_d)$ such that 
\begin{enumerate}
\item[(i)] (Regularity) 
\begin{multline}\label{e4.36}
\sum_{i=1}^d\po_{x_i}\b_{ik}(u) \in L_\loc^2((0,\infty)\X \GB_d),\\ \text{for $k=1,\cdots,d$, for $\b_{ik}(u)=\int^u\s_{ik}(v)\,dv$}.
\end{multline}

\item[(ii)] (Chain Rule)  For any function $\psi\in C(\R)$ and any $k=1,\cdots, d$ the following chain rule holds:
\begin{multline}\label{e4.38}
\sum_{i=1}^d\po_{x_i}\b_{ik}^\psi(u)=\psi(u)\sum_{i=1}^d\po_{x_i}\b_{ik}(u)\in L_\loc^2((0,\infty)\X \GB_d),\\ \quad\text{for $k=1,\cdots,d$, for $(\b_{ik}^\psi)'=\psi\b_{ik}'$}.
\end{multline}
\item[(iii)] (Entropy Inequality) For any convex $C^2$ function $\eta:\R\to\R$, and ${\bfq}'(u)=\eta'(u)\bff'(u)$, $r_{ij}'(u)=\eta'(u)a_{ij}(u)$, we have
\begin{equation}\label{e4.39}
\po_t\eta(u)+\nabla_x\cdot \bfq(u)-\sum_{ij=1}^d \po_{x_ix_j}^2r_{ij}(u)
\le -\eta''(u)\sum_{k=1}^d\left(\sum_{i=1}^d\po_{x_i}\b_{ik}(u)\right)^2,
\end{equation}
in the sense of the distributions in $(0,\infty)\X\GB_d$. 
\item[(iv)] (Initial Condition) 
\begin{equation}\label{e4.40}
\lim_{t\to0+}\intl_{\GB_d}|u(t,\om)-u_0(\om)|\,d\mm(\om)=0.
\end{equation}
\end{enumerate}
\end{definition}

\begin{theorem}[Existence]\label{T:3.1} There exists an entropy solution to the problem \eqref{e4.31}, \eqref{e4.32}.
\end{theorem}

\begin{proof}  Since $u_0\in L^\infty(\GB_d)$, then it is equivalent, via Gelfand transform, to a function, which we also denote $u_0$, in $\BAP(\R^d)\cap L^\infty(\R^d)$.
Let $u(t,x)$ be the entropy solution of \eqref{e1.1},\eqref{e1.2} with initial data $u_0$. First, applying  \eqref{e1.9} with $\eta(u)=\frac12 u^2$ to a suitable smooth approximation of  
$\varphi=\chi_{{}_{(0,T)\X \I_R}}$, for any $R, T>0$, we obtain that 
$$
\int_{(0,T)\X\I_R}\sum_{k=1}^d\left(\sum_{i=1}^d\po_{x_i}\b_{ik}(u)\right)^2\,dx\,dt\le C (R^d+ TR^{d-1}),
$$
for some $C>0$ independent of $R,T$, which gives
\begin{equation}\label{e3.cor1}
\limsup_{R\to\infty}\frac1{R^d}\int_0^T\int_{\I_R}\sum_{k=1}^d\left(\sum_{i=1}^d\po_{x_i}\b_{ik}(u)\right)^2\,dx\,dt\le C, \quad\text{for any $T>0$.}
\end{equation}
Let us denote $B_k=(\b_{1k}(u),\cdots,\b_{dk}(u))$ and $h_k=\div B_k$.  Let $\varphi\in C_c^\infty ((0,\infty)\X\GB_d)$, and we also denote by $\varphi$ its image by the inverse Gelfand transform defined on $(0,\infty)\X\R^d$. Since $\div (B_k \varphi) \in L_\loc^2((0,\infty)\X\R^d)$, we have that Gauss-Green theorem holds in the space variable. Therefore, we have
\begin{equation}\label{e3.cor2}
\frac1{R^d}\int_0^T\int_{\I_R} h_k\varphi(t,x)\,dx\,dt=-\frac1{R^d}\int_0^T\int_{\I_R} B_k\cdot \nabla\varphi\,dx\,dt+ O(\frac1R).
\end{equation}
 The right-hand side of the above equation has a well defined limit  when $R\to\infty$, and so does the left-hand side. Thus, the functional
 $$
 \la T_k, \varphi\ra :=\lim_{R\to\infty}\frac1{R^d}\int_0^T\int_{\I_R}h_k\varphi\,dx\,dt
 $$
 is well defined and, by \eqref{e3.cor1}, we get
 $$
 |\la T_k,\varphi\ra|^2\le C \|\varphi\|_{L^2((0,T)\X\GB_d)}^2.
 $$
 Hence, by an usual density argument, $T_k$ is a continuous linear functional over $L^2((0,T)\X\GB_d)$ and so $T_k$ may be represented by a function in $L^2((0,T)\X\GB_d)$, for
 $k=1,\cdots, d$.  Now, also from  \eqref{e3.cor2} we deduce
 $$
 \la T_k,\varphi\ra =- \la B_k, \nabla \varphi\ra_{L^2((0,T)\X\GB_d)}.
 $$
 Therefore, $T_k=h_k$ in the sense of the distributions in $(0,T)\X\GB_d$, and so $h_k\in L^2((0,T)\X\GB_d)$, $k=1,\cdots.d$.for all $T>0$, and so \eqref{e4.36} is verified.

Applying \eqref{e1.9} to a test function $\psi(t,x)= \phi(t,x)\frac1{R^d}\varphi(\frac{x}{R})$, where  $0\le \hat\phi\in C_c^\infty((0,\infty)\X\GB_d)$ and $\varphi\in C_c^\infty(\R^d)$, with $\varphi\ge 0$, $\int_{\R^d}\varphi\,dx=1$, we obtain
\begin{align*}
&\frac1{R^d}\int\limits_{\R_+^{d+1}}\left\{\eta(u)\phi_t+\bfq(u)\cdot\nabla\phi+\sum_{i,j=1}^d r_{ij}(u)\po_{x_ix_j}^2\phi\right.\\
&\left.-\eta''(u)\sum_{k=1}^d\left(\sum_{i=1}^d\po_{x_i}\b_{ik}(u)\right)^2\phi\right\}\varphi(\frac{x}{R})\,dx\,dt\\
&+\frac1{R^{d+1}}\int_{\R_+^{d+1}} \{\bfq(u)\cdot\nabla_y\varphi(\frac{x}{R})\phi+2\sum_{i,j=1}^d r_{ij}(u)\po_{x_i}\phi\po_{y_j}\varphi(\frac{x}{R})\}\,dx\,dt\\
&+\frac1{R^{d+2}}\int_{\R_+^{d+1}}\sum_{i,j=1}^d r_{ij}(u)\phi\po_{y_iy_j}^2\varphi(\frac{x}{R})\,dx\,dt\ge0.
\end{align*}
We make $R\to\infty$ and observe that the two last lines of the above inequality vanish as $R\to\infty$, while the first two give, as $R\to\infty$, using \eqref{e4.36}, 
\begin{multline}\label{e4.41}
\int\limits_{\R_+\X\GB_d}\left\{\eta(u)\phi_t+\bfq(u)\cdot\nabla\phi+\sum_{i,j=1}^d r_{ij}(u)\po_{x_ix_j}^2\phi\right.\\ 
\left.-\eta''(u)\sum_{k=1}^d\left(\sum_{i=1}^d\po_{x_i}\b_{ik}(u)\right)^2\phi\right\}\,d\mm(\om)\,dt\ge0,
\end{multline} 
which is \eqref{e4.39} in the sense of the distributions in $(0,\infty)\X\GB_d$.  

 Now, \eqref{e4.40} follows directly from Lemma~\ref{L:4.2}.   As to the chain rule \eqref{e4.38}, it follows directly from \eqref{e1.8} and the already proved \eqref{e4.39}. This concludes the proof. 

\end{proof}

For the proof of the uniqueness of the entropy solution to \eqref{e4.31},\eqref{e4.32} we use the following lemma, which is a trivial extension to $(0,\infty)\X\GB_d$ of
the corresponding fact in $(0,\infty)\X\R^d$, whose detailed proof is given in \cite{BK}, which in turn extends to the anisotropic case a fundamental trick first proved in \cite{Ca}. Let us denote
$$
\begin{aligned}
& F(u,v)=\sgn(u-v)(\fbf(u)-\fbf(v)),\\
&\Bbf(u,v)=\sgn(u-v)(\Abf(u)-\Abf(v)).
\end{aligned}
$$

\begin{lemma}\label{L:3.4} Let $\xi(t,\om,s,\z)$ be a nonnegative function in $C^\infty(Q\X Q)$, $Q=(0,\infty)\X\GB_d$, such that:
\begin{align*}
&(t, \om)\mapsto \xi(t,\om,s,\z)\in C_c^\infty(Q) \quad \text{for every $(s,\z)\in Q$},\\
&(s,\z)\mapsto\xi(t,\om,s,\z)\in C_c^\infty(Q)\quad \text{for every $(t,\om)\in Q$}.
\end{align*}
Let $u(t,\om)$ and $v(s,\z)$ be two entropy solutions to \eqref{e4.31},\eqref{e4.32}.  Then, we have
\begin{gather}
\iint_{Q\X Q} |u-v|(\xi_t+\xi_s)\,d\mm(\om)\,d\mm(\z)\notag\\ 
+\iint_{Q\X Q}F(u,v)\cdot(\nabla_x\xi+\nabla_y\xi)\,d\mm(\om)\,d\mm(\z)\label{e4.43}\\
+\iint_{Q\X Q} \Bbf(u,v)\cdot(\nabla_{x}^2\xi+\nabla_{xy}^2\xi+\nabla_{yx}^2\xi+\nabla_{y}^2\xi)\,d\mm(\om)\,d\mm(\z)\ge 0.\notag
\end{gather}
\end{lemma}

\begin{proof} The proof follows the same lines as the proof of the uniqueness theorem in \cite{BK}.  In  \eqref{e4.39},  we take  $\eta(v)=\eta_\ve(v)$ as a suitable approximation of $|v|$, say, $\eta_\ve(v)=\int_0^v\sgn_\ve(s)\,ds$, with
$$
\sgn_\ve(v)=\begin{cases} -1, &v<-\ve,\\ \sin(\frac{\pi}{2\ve}v),& |v|\le\ve,\\ 1, & v>\ve.
\end{cases}
$$
We  get
 \begin{multline}\label{e4.43'}
\int\limits_{\R_+\X\GB_d}\left\{|u-k|\xi_t+F(u,k)\cdot\nabla_x\xi+\sum_{i,j=1}^d r_{ij}(u,k)\po_{x_ix_j}^2\xi\right\}\, d\mm(\om)\,dt\\ 
-\int\limits_{\R_+\X\GB_d}\left\{\sgn_\ve'(u-k)\sum_{k=1}^d\left(\sum_{i=1}^d\po_{x_i}\b_{ik}(u)\right)^2\xi\right\}\,d\mm(\om)\,dt\ge I_\ve(u,k,\xi_t,\nabla_x\xi, D_x^2\xi),
\end{multline}  
where $\Bbf(u,v)=(r_{ij}(u,v))_{i,j=1}^d$ and $I_\ve(u,k,\xi_t,\nabla_x\xi, D_x^2\xi)$ is the difference between the first integral and the corresponding one with $\eta_\ve(u-k)$, $F_\ve(u,k)$, $r_{\ve ij}(u,k)$, instead of $|u-.k|$, $F(u,k)$ and $r_{ij}(u,k)$,
where $F_\ve(u,k)$ and $r_{\ve ij}(u,k)$ are the entropy flux and the viscosity matrix associated with $\eta_\ve$, respectively. We then use Kruzhkov's doubling of variables method (see \cite{Kr}), making $k=v(s,\z)$ and integrating with respect to $(s,\z)\in Q$ to get
\begin{multline}\label{e4.44}
\int\limits_{Q\X Q}\left\{|u-v|\xi_t+F(u,v)\cdot\nabla_x\xi+\sum_{i,j=1}^d r_{ij}(u,v)\po_{x_ix_j}^2\xi\right\}\,d\mm(\om)\,dt\,d\mm(\z)\,ds \\ 
-\int\limits_{Q\X Q}\left\{\sgn_\ve'(u-v)\sum_{k=1}^d\left(\sum_{i=1}^d\po_{x_i}\b_{ik}(u)\right)^2\xi\right\}\,d\mm(\om)\,dt\,d\mm(\z)\,ds\\
\ge \int_Q I_\ve(u,v,\xi_t,\nabla\xi,D_x^2\xi)\,d\mm(\z)\,ds.
\end{multline}    
We proceed in the same way with the inequality for $v(s,\z)$ analog to \eqref{e4.43'}, this time making $k=u(t,\om)$ and then integrating with respect to $(t,\om)\in Q$ to get
\begin{multline}\label{e4.45}
\int\limits_{Q\X Q}\left\{|u-v|\xi_s+F(u,v)\cdot\nabla_y\xi+\sum_{i,j=1}^d r_{ij}(u,v)\po_{y_iy_j}^2\xi\right\}\,d\mm(\om)\,dt\,d\mm(\z)\,ds\\ 
-\int\limits_{Q\X Q}\left\{\sgn_\ve'(u-v)\sum_{k=1}^d\left(\sum_{i=1}^d\po_{y_i}\b_{ik}(v)\right)^2\xi\right\}\,d\mm(\om)\,dt\,d\mm(\z)\,ds\\
\ge  \int_Q I_\ve(v,u,\xi_t,\nabla\xi,D_x^2\xi)\,d\mm(\om)\,dt.
\end{multline}     
We sum these inequalities to get 
\begin{multline}\label{e4.46}
\int\limits_{Q\X Q}|u-v|(\xi_t+\xi_s)+F(u,v)\cdot(\nabla_x+\nabla_y)\xi 
+\sum_{i,j=1}^d r_{ij}(u,v)(\po_{x_ix_j}^2\xi+\po_{y_iy_j}^2\xi)\\ 
-\sgn_\ve'(u-v)\sum_{k=1}^d\left(\left(\sum_{i=1}^d\po_{x_i}\b_{ik}(u)\right)^2+\left(\sum_{i=1}^d\po_{y_i}\b_{ik}(v)\right)^2\right)\xi \,d\mm(\om)\,dt\,d\mm(\z)\,ds\\
\ge  \int_Q I_\ve(u,v,\xi_t,\nabla\xi,D_x^2\xi)\,d\mm(\z)\,ds+ \int_Q I_\ve(v,u,\xi_t,\nabla\xi,D_x^2\xi)\,d\mm(\om)\,dt,
\end{multline}   
We then add the mixed derivatives $r_{ij}(u,v)(\po_{x_iy_j}^2\xi+\po_{y_ix_j}^2\xi)$ in the third term and use the trivial inequality
\begin{equation*}
\left(\sum_{i=1}^d\po_{x_i}\b_{ik}(u)\right)^2+\left(\sum_{i=1}^d\po_{y_i}\b_{ik}(v)\right)^2
\ge 2\sum_{i,j=1}^d \po_{x_i}\b_{ik}(u)\po_{y_j}\b_{jk}(v),
\end{equation*}
in the fourth term to get
\begin{multline}\label{e4.47}
\int\limits_{Q\X Q}|u-v|(\xi_t+\xi_s)+F(u,v)\cdot(\nabla_x+\nabla_y)\xi \\
+\sum_{i,j=1}^d r_{ij}(u,v)(\po_{x_ix_j}^2\xi+\po_{x_iy_j}^2\xi+\po_{y_ix_j}^2\xi+ \po_{y_iy_j}^2\xi)\\ 
-2\,\sgn_\ve'(u-v)\sum_{k=1}^d\sum_{i,j=1}^d \po_{x_i}\b_{ik}(u)\po_{y_j}\b_{jk}(v)\xi \\
-\sum_{i,j=1}^dr_{ij}(u,v)(\po_{x_iy_j}^2\xi+\po_{y_ix_j}^2\xi)\,d\mm(\om)\,dt\,d\mm(\z)\,ds\\
\ge  \int_Q I_\ve(u,v,\xi_t,\nabla\xi,D_x^2\xi)\,d\mm(\z)\,ds+ \int_Q I_\ve(v,u,\xi_t,\nabla\xi,D_x^2\xi)\,d\mm(\om)\,dt .
\end{multline}   
As in \cite{BK}, we conclude the proof by using twice the chain rule followed by integration by parts in the third line of \eqref{e4.47}. Namely, using the chain rule first in the $y_j$ partial derivatives followed by integration by parts in the same derivatives, this line becomes
$$
\limsup_{\ve\to0}2\sum_{k=1}^d\sum_{i,j=1}^d \po_{x_i}\b_{ik}(u)\b_{jk}^{\sgn_\ve'(u-v)}(v)\po_{y_j}\xi 
$$
where $\frac{d}{dv}\b_{jk}^{\sgn_\ve'(u-v)}(v)=\sgn_\ve'(u-v)\s_{jk}(v)$,  so, we may take
$$
\b_{jk}^{\sgn_\ve'(u-v)}(v)=-\sgn_\ve(u-v)\s_{jk}(v)+\int_u^v\sgn_\ve(u-s)\s_{jk}'(s)\,ds.
$$
We use again the chain rule, now in the $x_i$ partial derivatives, followed by integration by parts in the same derivatives, so that the third line becomes
$$
-2\sum_{k=1}^d\sum_{i,j=1}^d \left(\int_{u}^v\s_{ik}(s)\b_{jk}^{\sgn_\ve'(s-v)}(v)\,ds\right)\po_{x_iy_j}^2\xi. 
$$
Now, we observe that 
$$
\lim_{\ve\to0}\b_{jk}^{\sgn_\ve'(u-v)}(v)=-\sgn(u-v)\s_{jk}(u),
$$
so, taking the limit when $\ve\to0$, the third line becomes
$$
2\sum_{k=1}^d\sum_{i,j=1}^d \left(\int_{v}^u \sgn(s-v)\s_{ik}(s)\s_{jk}(s)\,ds\right)\po_{x_iy_j}^2\xi= 2\sum_{i,j=1}^dr_{ij}(u,v)\po_{x_iy_j}^2\xi,
$$
which cancels with the fourth line  in \eqref{e4.47} since $r_{ij}(u,v)=r_{ji}(u,v)$. Finally, on the other hand, it is straightforward to verify that 
$$
 \int_Q I_\ve(u,v,\xi_t,\nabla\xi,D_x^2\xi)\,d\mm(\z)\,ds+ \int_Q I_\ve(v,u,\xi_t,\nabla\xi,D_x^2\xi)\,d\mm(\om)\,dt\to 0, \quad\text{as $\ve\to0$},
 $$
concluding the proof of  the lemma.

\end{proof}

\begin{theorem}[Uniqueness]\label{T:3.2} There is at most one entropy solution to \eqref{e4.31},\eqref{e4.32}. More specifically, given two entropy solutions \eqref{e4.31},\eqref{e4.32}, $u(t,\om),\ v(t,\om)$, with initial functions $u_0(\om),\ v_0(\om)$, for a.e.\ $t>0$, 
\begin{equation}\label{eT3.2}
\int_{\GB_d}|u(t,\om)-v(t,\om)|\,d\mm(\om)\le \int_{\GB_d}|u_0(\om)-v_0(\om)|\,d\mm(\om).
\end{equation}
\end{theorem}

\begin{proof}   In \eqref{e4.43} we take 
\begin{equation}\label{e4.48}
\xi=\d_\nu(t-s) \tilde \rho_\a(\om-\z)\phi(t,\om)
\end{equation}
where $\d_\nu$ is as in the proof of Proposition~\ref{P:1.1}, $\phi\in C_c^\infty((0,\infty)\X\GB_d)$, and  $\tilde \rho_\a$ is obtained by  suitably mollifying an  approximate identity $\rho_\a$ in $\GB_d$ with $\a$ running through a countable basis of neighborhoods of 0 in the topology generated by $\A_{u,v}$ and for simplicity we take $\a\in\N$. More specifically, we may define $\tilde \rho_\a$ as a smooth almost periodic function on $\R^d$ by setting $\tilde \rho_\a=c_\a \rho_\a*\psi_\a$, where $\psi_\a$ is a standard mollifying kernel in $\R^d$ such that $\|\rho_\a-\rho_\a*\psi_\a\|_\infty<\ve_\a$, with $\ve_\a\to0$, as $\a\to\infty$, and $c_\a$ is a normalizing factor such that $\int_{\GB_d}\tilde \rho_\a(\om)\,d\mm(\om)=1$.  
In particular,  $\tilde\rho_\a*g\to g$ as $\a\to\infty$, in the $\sup$-norm, for all $g\in\A_{u,v}$. We observe that for any $g\in C^2(\GB_d)$, trivially we have
\begin{multline*}
\po_{x_i x_j}^2g(\om-\z)+\po_{x_iy_j}^2g(\om-\z)+\po_{y_ix_j}^2g(\om-\z)+\po_{y_iy_j}^2g(\om-\z)\\
=(\po_{x_i}+\po_{y_i})(\po_{x_j}+\po_{y_j})g(\om-\z)=0.
\end{multline*} 
Therefore, when we take $\xi$, defined  as in \eqref{e4.48}, in \eqref{e4.43} and make $\nu\to\infty$ and $\a\to\infty$, we end up with
\begin{equation}
\iint_{Q} |u-v|\phi_t+ F(u,v)\cdot\nabla_x\phi 
+ \Bbf(u,v)\cdot \nabla_{x}^2\phi\,d\mm(\om)\ge 0.\label{e4.49}
\end{equation} 
We then take $\phi(t,\om)=\chi_\nu(t)$ in \eqref{e4.49}, with $\chi_\nu$ as in the proof of Proposition~\ref{P:1.1}, make $\nu\to\infty$ and then make $t_0\to0$, using \eqref{e4.40}, to
finally get \eqref{eT3.2}.

\end{proof}

We remark that, from \eqref{e4.49}, it also follows, by using the same kind of test function, $\phi(t,\om)=\chi_\nu(t)$, with $\chi_\nu'(t)=\d_\nu(t-t_1)-\d_\nu(t-t_2)$, as in the proof of Proposition~\ref{P:1.1},  that for all $0\le t_1\le t_2$,
\begin{equation}\label{eT3.2'}
\int_{\GB_d}|u(t_2,\om)-v(t_2,\om)|\,d\mm(\om)\le \int_{\GB_d}|u(t_1,\om)-v(t_1,\om)|\,d\mm(\om),
\end{equation}
where we use the fact that $u\in C([0,\infty),L^1(\GB_d))$, which is justified by Theorem~\ref{T:3.2} and Lemma~\ref{L:4.2}.

\begin{lemma}\label{L:3.5} Let $u(t,\om)$ be the entropy solution of \eqref{e4.31},\eqref{e4.32}. Then, for $0\le t_1\le t_2$, 
\begin{equation}\label{e4.50}
\int_{\GB_d}|u(t_2,\om)|^2\,d\mm(\om)\le \int_{\GB_d}|u(t_1,\om)|^2\,d\mm(\om).
\end{equation}
\end{lemma}

\begin{proof} By uniqueness, $u(t,\om)$ is the Gelfand transform of the entropy solution of \eqref{e1.1},\eqref{e1.2}, whose initial data is the inverse Gelfand transform of $u_0(\om)$. In particular, $u\in C([0,\infty); L^1(\GB_d))$. Now, we apply \eqref{e4.39}, with $\eta(u)= u^2$,  in a test function $\phi=\chi_\nu(t)$ with $\chi_\nu$ as in the proof of Proposition~\ref{P:1.1}, to obtain
$$
\int_0^\infty\int_{\GB_d} u^2(t,\om) (\d_\nu(t-t_1)-\d_\nu(t-t_2))\,d\mm(\om)\,dt\ge 0,
$$
where $\d_\nu$ is as in the proof of Proposition~\ref{P:1.1}. Then, if $t_1$ and $t_2$ are Lebesgue points of $\int_{\GB_d}u^2(t,\om)\,d\mm(\om)$, making $\nu\to\infty$, we get \eqref{e4.50}. Now, since $u\in C([0,\infty); L^1(\GB_d))$, we conclude that \eqref{e4.50} holds for all $0\le t_1\le t_2$. 
\end{proof}

\section{Proof of the decay property}\label{S:4}

In this section we prove the decay property \eqref{e1.13}. We first remark that from Theorem~\ref{T:3.2} and the proof of Theorem~\ref{T:3.1} it follows that the Gelfand transform of the entropy solution of \eqref{e1.1}-\eqref{e1.2} is the unique solution of \eqref{e4.31}-\eqref{e4.32}. It is also clear that the decay to the mean value, in the sense of the $L^1(\GB_d)$-norm, of  the solution of \eqref{e4.31}-\eqref{e4.32} is equivalent to the decay to the mean value of its inverse Gelfand transform, that is, the entropy solution of \eqref{e1.1}-\eqref{e1.2}, in the sense of \eqref{e1.13}.  So, it suffices to prove the decay to the mean value of the solution of \eqref{e4.31}-\eqref{e4.32} in the sense of the $L^1(\GB_d)$-norm.   Our proof is strongly motivated by the proof of the corresponding decay property for periodic entropy solutions in \cite{CP2}. 
As in  \cite{CP2}, we assume, without loss of generality,  that $\Me(u_0)=\int_{\GB_d}u_0(\om)\,d\mm(\om)=0$ and so $\int_{\GB_d}u(t,\om)\,d\mm(\om)=0$, for all $t\ge0$; otherwise we may replace $u(t,\om)$ by $u(t,\om)-\bar u$, $\bff(u)$ by $\bff(u+\bar u)$, and $\bfA(u)$ by $\bfA(u+\bar u)$ in \eqref{e1.1} and \eqref{e4.31}, and $u_0$ by $u_0-\bar u$ in \eqref{e1.2} and \eqref{e4.32}. 

First we observe that \eqref{e4.39} can be rewritten in the form
\begin{multline}\label{e5.1}
\po_t\eta(u)+\nabla_x\cdot \bfq(u) -\nabla_x^2: \bfA^{\eta'}(u)=-(m^{\eta''}+n^{\eta''}),\\
\text{for any convex $C^2$ function $\eta:\R\to\R$ and ${\bfq}'(u)=\eta'(u)\bff'(u)$},
\end{multline}
in the sense of the distributions on $(0,\infty)\X\GB_d$, where $\bfA^{\eta'}(u)=\int^u \eta'(\xi)A(\xi)\,d\xi$, 
\begin{multline}\label{e5.2}
m^{\eta''}(t,\om)=\int_{\R}\eta''(\xi) \,dm(t,\om,\xi),\\ \text{with $m(t,\om,\xi)$ a nonnegative measure on $(0,\infty)\X\GB_d\X\R$},
\end{multline}
and
\begin{equation}\label{e5.3}
n^{\eta''}(t,\om)=\int_{\R} \eta''(\xi) \,d n(t,\om,\xi)
\end{equation}
for $n(t,\om,\xi)$ the measure on $(0,\infty)\X\GB_d\X\R$ defined as
\begin{equation}\label{e5.4}
n(t,\om,\xi):=\d(\xi-u(t,\om))\sum_{k=1}^d\left(\sum_{i=1}^d\po_{x_i}\b_{ik}(u(t,\om))\right)^2.
\end{equation}

Let us recall the kinetic $\chi$ on $\R^2$ defined by:
\begin{equation}\label{e5.5}
\chi(\xi;u)=\begin{cases} 1 &\text{for $0<\xi<u$}\\ -1 &\text{for $u<\xi<0$}\\ 0 &\text{otherwise}
\end{cases}.
\end{equation}

The relation $S(u)=\int_{\R}S'(\xi)\chi(\xi;u)\,d\xi$, valid for all Lipschitz $S(u)$ with $S(0)=0$, yields the following kinetic equation for $\chi(\xi;u(t,\om))$, where  $u(t,\om)$  is the entropy solution of \eqref{e4.31},\eqref{e4.32}, which is equivalent to the entropy identity \eqref{e5.1}:
\begin{equation}\label{e5.6}
\po_t\chi(\xi;u)+a(\xi)\cdot\nabla_x\chi(\xi;u)-\nabla_x\cdot(A(\xi)\nabla_x\chi(\xi;u))=\po_\xi(m+n)(t,\om,\xi)
\end{equation}
in the sense of the distributions on $(0,\infty)\X\GB_d\X\R$, with initial condition
\begin{equation}\label{e5.7}
\chi(\xi;u(t,\om))|_{t=0}=\chi(\xi;u_0(\om)),
\end{equation}
and $n(t,\om,\xi)$ is defined by \eqref{e5.4}. 

We observe that  by Lemma~\ref{L:3.5} the function
$$
I(t):=\int_{\GB_d}|u(t,\om)|^2\,d\mm(\om)
$$
is a non-increasing, bounded function and so the following limit exists:
\begin{equation}\label{e5.8}
\lim_{t\to\infty} I(t)=I(\infty)=:I_\infty\in[0,\infty).
\end{equation}
We define the translation sequence
$$
v^k(t,\om):=u(t+k,\om), \quad k\in\N,
$$
defined in $[-k,\infty)\X\GB_d$.  Trivially,
\begin{equation}\label{e5.9}
\|v^k(t,\cdot)\|_\infty=\|u(t+k,\cdot)\|_\infty\le \|u_0\|_\infty,
\end{equation}
and $\chi(\xi;v^k(t,\om))$ satisfies
\begin{equation}\label{e5.10}
\po_t\chi(\xi;v^k)+a(\xi)\cdot\nabla_x\chi(\xi;v^k)-\nabla_x\cdot(A(\xi)\nabla_x\chi(\xi;v^k))=\po_\xi(m^k+n^k)(t,\om,\xi)
\end{equation}
as distributions in $(-k,\infty)\X\GB_d\X\R$, with $(m^k+n^k)(t,\om,\xi)=(m+n)(t+k,\om,\xi)$. 
     
We have the following compactness property for the sequence $v^k(t,\om)$.

\begin{lemma}\label{L:5.1} There exists a subsequence $\{v^{k_j}\}_{j=1}^\infty\subset \{v^k\}_{k=1}^\infty$ and $v(t,\om)\in L^\infty(\R\X\GB_d)$, with $\int_{\GB_d}v(t,\om)\,d\mm(\om)=0$, for all  $t\in\R$, such that 
\begin{equation}\label{e5.11}
v^{k_j}(t,\om)\to v(t,\om)\quad \text{a.e.\ $(t,\om)\in\R\X\GB_d$ as $j\to\infty$}
\end{equation}
and, correspondingly,
\begin{equation} \label{e5.12}
\chi(\xi;v^k(t,\om))\to\chi(\xi;v(t,\om))\quad \text{a.e.\ $(t,\om,\xi)\in \R\X\GB_d\X\R$ as $j\to\infty$}.
\end{equation}
Moreover, $\chi(\xi;v)$ is a solution in $\DD'(\R\X\GB_d\X\R)$, the space of distributions on $\R\X\GB_d$,  of
\begin{equation}\label{e5.13}
\po_t\chi+a(\xi)\cdot\nabla_x\chi-\nabla_x\cdot(A(\xi)\nabla_x\chi)=0.
\end{equation}
In particular,
\begin{equation}\label{e5.14}
\int_{\GB_d}|v(t,x)|^2\,dx= I_\infty\in [0,\infty),\quad \text{for a.e.\ $t\in\R$},
\end{equation}

\end{lemma}

\begin{proof}  Let $N_0\in\N$   be given. For $k\ge N_0$, we have from \eqref{eT3.2'}, for $t\ge -N_0$, 
\begin{equation}\label{e5.15}
\int_{\GB_d}|v^k(t,\om+\z)-v^k(t,\om)|\,d\mm(\om)\le \int_{\GB_d}|u_0(\om+\z)-u_0(\om)|\,d\mm(\om),
\end{equation}
for any $\z\in\GB_d$, where we also use the fact that if $u(t,\om)$ is an entropy solution of \eqref{e4.31},\eqref{e4.32}, then $u(t,\om+\z)$ is also an entropy solution of \eqref{e4.31} with initial data $u_0(\om+\z)$, which may be trivially checked. For each fixed $t\ge-N_0$, the inequality \eqref{e5.15} implies the compactness in $L^1(\GB_d)$ of
the family $\{v^k(t,\cdot)\}_{k\in\N}$, which can be easily proven by means of an approximate identity $\rho_\a$ as  in the proof of Lemma~\ref{L:3.1} (see \eqref{e4.26}). In particular, for $t=-N_0$, we can extract a subsequence $\{v^{k(N_0,j)}\}$ such that $\{v^{k(N_0,j)}(-N_0,\cdot)\}$ converges in $L^1(\GB_d)$. But, for any $N\in\N$, we have 
\begin{multline*}
\int_{-N_0}^N \int_{\GB_d}|v^{k(N_0,l)}(t,\om) -v^{k(N_0,m)}|\,d\mm(\om)\,dt \\ \le (N+N_0)\int_{\GB_d}|v^{k(N_0,l)}(-N_0, \om)-v^{k(N_0,m)}(-N_0,\om)|\,d\mm(\om),
\end{multline*}
 which implies that $\{v^{k(N_0,j)}\}_{j\in\N}$ converges in $L^1([-N_0,N]\X\GB_d)$, for any $N\in\N$. We may now, in an iterative way, making $N_0=1,2,\cdots$, extract  successive subsequences $\{v^{k(N_0,j)}\}$, each converging in $L_{\loc}^1([-N_0,\infty)\X\GB_d)$, respectively, such that $\{v^{k(m+1,j)}\}\subset \{v^{k(m,j)}\}$, for all $m\in\N$. 
We then use a diagonal process, to define a  subsequence $\{v^{k_j}\}_{j\in\N}$ converging in $L_{\loc}^1(\R\X\GB_d)$ to a certain $v\in L^\infty(\R\X\GB_d)$. Extracting a further subsequence, if necessary, we obtain that $v^{k_j}$ converges a.e.\ to $v$ in $\R\X\GB_d$, which is \eqref{e5.11}.  Assertion \eqref{e5.12} follows from \eqref{e5.11} by a well known  property of $\chi$-functions. 

Now, multiplying \eqref{e5.10} by $\xi$ and integrating on $(t,x,\xi)\in (-T,T)\X\GB_d\X\R$, for any $0<T\le k$, we obtain
\begin{equation}\label{e5.16}
\int_{\R}\int_{-T}^T\int_{\GB_d}(m^k+n^k)(t,\om,\xi)\,d\mm(\om)\,dt\,d\xi \le \frac12(I(k-T)-I(k+T))\le \frac12\|u_0\|_\infty^2.
\end{equation}
So, the nonnegative measures sequences $\{m^k(t,\om,\xi)\}$, $\{n^k(t,\om,\xi)\}$ are uniformly bounded in $k$ over $(-T,T)\X\GB_d\X\R$. Therefore, there exist a subsequence, which we still label $k_j$, and measures $M_1(t,\om,\xi)$ and $M_2(t,\om,\xi)$ such that 
\begin{multline*}
m^{k_j}(t,\om,\xi)\wto M_1(t,\om,\xi)\ge0,\; n^{k_j}(t,\om,\xi)\wto M_2(t,\om,\xi)\ge0, \\ \text{weakly in $\M((-T,T)\X\GB_d\X \R)$, as $j\to\infty$, for all $T>0$}.
\end{multline*}
Now, since $I(t)$ converges, we have that
$$
I(k-T)-I(T+k)\to 0,\quad\text{as $k\to\infty$}.
$$
Hence, \eqref{e5.16} implies that 
\begin{equation}\label{e5.17}
M_1(\R\X\GB_d\X\R)=0,\quad M_2(\R\X\GB_d\X\R)=0.
\end{equation}
Moreover, taking $k=k_j$ and making $j\to\infty$ in \eqref{e5.10}, we conclude that $\chi(\xi; v(t,\om))$ is an function in $L^\infty(\R\X\GB_d\X\R)$ satisfying \eqref{e5.13} in 
$\DD'(\R\X\GB_d\X \R)$. 

Finally, multiplying \eqref{e5.13} by $\xi$ and then integrating $d\mm(\om)\,d\xi$ over $\GB_d\X\R$, we get \eqref{e5.14}.
  
\end{proof}
 
About the limit function $v(t,\om)$ we also have the following:

\begin{lemma} \label{L:5.2}  The limit function $v(t,\om)$ also satisfies the properties:
\begin{enumerate}
\item[(i)] $v\in C(\R; L^1(\GB_d))$. 

\item[(ii)] In particular, $\Lambda_v:=\{\l\in \R\,:\, \exists t\in\R,\; \l\in\Sp(v(t,\cdot))\}$ is a countable set, and so is $\Gr(v)$, the smallest additive group generated by $\Lambda_v$.  Actually, $\Lambda_v=\Lambda_u$. Further, for all $(t,\xi)\in\R^2$ we have $\Sp(\chi(\xi;v(t,\cdot)))\subset \Gr(v)$. 

\item[(iii)] The family $\{v(t,\cdot)\}_{t\in\R}$ is relatively compact in $L^1(\GB_d)$.

\item[(iv)]   If we write 
\begin{equation}\label{e5.18}
v(t,\om)=\sum_{\l\in\Lambda_v} a_\l(t) e^{2\pi i\l\cdot \om},
\end{equation}
with equality in the sense of $L^2(\GB_d)$, so that
\begin{equation}\label{e5.19}
\sum_{\l\in\Lambda_v}|a_\l(t)|^2=I_\infty,
\end{equation}
then, given $\ve>0$, there is a finite set $F_\ve\subset \Lambda_v$ such that
\begin{equation}\label{e5.20}
\sum_{\l\in\Lambda_v\setminus F_\ve} |a_\l(t)|^2<\ve,\quad \text{for all $t\in\R$}.
\end{equation}
\end{enumerate}
\end{lemma}

\begin{proof} 
To prove (i), we observe first that multiplying \eqref{e5.13} by $\eta'(\xi)$, with $\eta\in C^1(\R)$, integrate in $\xi$, to obtain
\begin{equation}\label{e5.21}
\eta(v)_t+\nabla_x\cdot\bfq(v)-\nabla_x^2:\bfA^{\eta'}(v)=0,
\end{equation}
in $\DD'(\R\X\GB_d)$, where $\bfq'(v)=\eta'(v)a(v)$, $(\bfA^{\eta'})'(v)=\eta'(v) A(v)$, recalling that $a(v)=\bff'(v)$, $A(v)=\bfA'(v)$. We apply \eqref{e5.13} to a test function
$\phi\in C_c^1(\R; C^2(\GB_d)\cap\A_v)$, to get
\begin{equation}\label{e5.22}
\int_{\R}\int_{\GB_d}\eta(v)\phi_t+\bfq(v)\cdot \nabla_x\phi+\bfA^{\eta'}(v):\nabla_x^2\phi\, d\mm(\om)\,dt=0.
\end{equation}
Let $t_0$ be a Lebesgue point of the function
\begin{equation}\label{e5.23}
t\mapsto \int_{\GB_d} \eta(v(t,\om))\phi(t,\om)\,d\mm(\om).
\end{equation}
Then, replacing $\phi$ in \eqref{e5.22} by $\tilde\phi(t,\om)=\chi_\nu(t)\phi(t,\om)$, where $\chi_\nu$ is as in the proof of Proposition~\ref{P:1.1}, with $t_1$ such that $\supp \phi\subset\{(-\infty,t_1)\X\GB_d\}$,  making $\nu\to\infty$, we get
\begin{multline}\label{e5.24}
\int_{t_0}^\infty\int_{\GB_d}\eta(v)\phi_t+\bfq(v)\cdot \nabla_x\phi+\bfA^{\eta'}(v):\nabla_x^2\phi\, d\mm(\om)\,dt\\+\int_{\GB_d}\eta(v(t_0,\om))\phi(t_0,\om)\,d\mm(\om)=0.
\end{multline} 
We may take $\phi$ as running through a dense subset of $C_c^1(\R;C^2(\GB_d)\cap\A_v)$, so that the set of points that are Lebesgue points of all the corresponding functions \eqref{e5.23} form a subset of $\R$  whose complement has measure zero.  Denoting $L^2(\GB_d;\A_v)$ the closure in $L^2(\GB_d)$ of $\A_v$, we see that the set of points that are Lebesgue points of \eqref{e5.23} for all $\phi$ in the referred dense subset of $C_c^1(\R; C^2(\GB_d)\cap\A_v)$, is also a set of Lebesgue points of \eqref{e5.23} for all $\phi\in C_c^1(\R; C^2(\GB_d)\cap\A_v)$, and it is also a set of Lebesgue points for all $\phi\in C_c^1(\R;L^2(\GB_d;\A_v))$. Now, any function in $C_c^1(\R;C^2(\GB_d))$ satisfies $\phi=\phi_1+\phi_2$, with $\phi_1\in C_c^1(\R; L^2(\GB_d;\A_v))$ and $\phi_2\in C_c^1(\R;L^2(\GB_d;\A_v)^\perp)$, where $L^2(\GB_d;\A_v)^\perp$ is the orthogonal complement of $L^2(\GB_d;\A_v)$ in $L^2(\GB_d)$, and, by orthogonality,  the Lebesgue points of \eqref{e5.23} for $\phi\in C_c^1(\R; C^2(\GB_d))$ are the Lebesgue points  for the corresponding $\phi_1\in C_c^1(\R;L^2(\GB_d;\A_v))$, since $\eta(v(t;\cdot))\in L^2(\GB_d;\A_v)$, for a.e.\ $t\in\R$. Therefore, we conclude that  \eqref{e5.22} holds for all 
$\phi\in C_c^1(\R; C^2(\GB_d))$ and  a.e.\ $t_0\in\R$. We can also take $\eta$ belonging to a dense subset of  $C^1([-M,M])$, $M=\|u_0\|_\infty$, and conclude that \eqref{e5.22} holds for all $\eta\in C^1([-M,M])$, for all  $\phi\in C_c^1(\R; C^2(\GB_d))$ and  a.e.\ $t_0\in\R$.

Now, from the discussion leading to \eqref{e5.17}, we see that 
\begin{equation}\label{e5.25}
\sum_{i=1}^d\po_{x_i}\b_{ik}(v)=0, \quad\text{a.e.\ in $\R\X\GB_d$, for $k=1,\cdots,d$}.
\end{equation}

More generally,  recalling the notation $\b^\psi(v)$ as meaning $(\b^\psi)'(v)=\psi(v)\b'(v)$, given any $\psi\in C(\R)$, 
\begin{equation}\label{e5.26}
\sum_{i=1}^d\po_{x_i}\b_{ik}^\psi(v)=0, \quad\text{a.e.\ in $\R\X\GB_d$, for $k=1,\cdots,d$}.
\end{equation}

Finally, we claim that for all $t_0$ for which \eqref{e5.24} holds, then we have
\begin{equation}\label{e5.27}
\lim_{t\to t_0+}\int_{\GB_d}|v(t,\om)-v(t_0,\om)|\,d\mm(\om)=0.
\end{equation}
Indeed, this follows in an standard way from \eqref{e5.24}, by first choosing $\phi$ of the form $\phi(t,\om)=\chi_\nu(t)\varphi(\om)$, for $\chi_\nu$ as above, but approaching the characteristic function of an interval $[t_*, t_1]$, with $t_*<t_0<t_1$, with $t_1$ belonging to the set of Lebesgue points obtained above,  and $\varphi\in C^2(\GB_d)$. Then, making $\nu\to\infty$ and then $t_1\to t_0+$, we deduce that
\begin{equation}\label{e5.28}
\lim_{t\to t_0+}\int_{\GB_d}\eta(v(t,\om))\varphi(\om)\, d\mm(\om) =\int_{\GB_d}\eta(v(t_0,\om))\varphi(\om)\,d\mm(\om).
\end{equation}
 Choosing, by approximation, $\eta(v)=|v-k|$, for $k\in\R$ arbitrary, and then extending \eqref{e5.28} to $\varphi\in L^1(\GB_d)$, we get for any simple function $\s(\om)$ we have
 \begin{equation}\label{e5.29}
\lim_{t\to t_0+}\int_{\GB_d}|v(t,\om)-\s(\om)|\, d\mm(\om) =\int_{\GB_d}|v(t_0,\om)-\s(\om)|\,d\mm(\om),
\end{equation}
 so that, choosing $\s(\om)$ as a sequence of simple functions converging to $v(t_0,\om)$ in $L^1(\GB_d)$ we arrive at \eqref{e5.27}. 
 
 The facts proved so far show that, for $t_0$ in a set of total measure in $\R$, $v(t_0+t,\om)$ is an entropy solution of  \eqref{e4.31},\eqref{e4.32} in $(0,\infty)\X\GB_d$, with initial
 data $v(t_0,\om)$.  Therefore, by Lemma~\ref{L:4.2},  $v\in C(\R; L^1(\GB_d))$, and so (i) is proved.
 
 The assertion (ii) follows immediately from (i) as in Lemma~\ref{L:4.3}. The fact that $\Lambda_v=\Lambda_u$ follows from the fact that, for a.e.\ $t\in\R$, $v(t,\cdot)$ is the limit in $L^1(\GB_d)$ of $v^{k_j}(t,\cdot)=u(t+k_j,\cdot)$, and so follows the equality. 
 
 As for (iii), we observe first that, for each $j$,  the family $\{v^{k_j}(t,\cdot)\}_{t\ge -k_j}$ coincides with the family $\{u(t,\cdot)\}_{t\ge0}$, which is compact. In particular,
 for each $t\in\R$, the sequence $\{v^{k_j}(t,\cdot)\}_{j\in\N}$ is contained in a fixed compact in $L^1(\GB_d)$. Since, for a.e.\ $t\in\R$, the sequence  $\{v^{k_j}(t,\cdot)\}$ converges in $L^1(\GB_d)$ to $v(t,\cdot)$, we conclude that the family $\{v(t,\cdot)\}_{t\in\R}$ is relatively compact.
 
 Finally, (iv) follows from (iii) as in Lemma~\ref{L:3.2}.

\end{proof}

The next  final step follows closely the lines in \cite{CP2}. The major problem in order to adapt the ideas in \cite{CP2} to the present Besicovitch almost periodic case is that, in order to apply condition \eqref{e1.4'}, it would be necessary to have the frequencies 
\begin{equation}\label{e5.30}
\k=2\pi\l, \quad \text{for $\l\in\Lambda_v$}, 
\end{equation}
satisfying $\k\ge\d_0$, for some $\d_0>0$. Since the $0$ frequency is excluded by the assumption that $\Me(u_0)=0$, in the periodic case this property is trivially satisfied by the fact that the set of $\k$'s coincides with the set of integer numbers multiplied by a constant. In the almost periodic case, although the set of frequencies is still countable, it may accumulate in 0 and so we would not have the mentioned property satisfied. Nevertheless, Lemma~\ref{L:5.2}~(iv) provides us with a way around this difficulty.

\begin{lemma}\label{L:5.3} We have $v(t,\om)=0$ for a.e.\ $(t,\om)\in\R\X\GB_d$.
\end{lemma}

\begin{proof} As in \cite{CP2}, we introduce a ``time truncation'' function $\phi(t)$, $0\le\phi(t)\le 1$, so that $\phi\chi$ belongs to $L^2(\R\X\GB_d\X\R)$, where $\chi(t,\om,\xi):=\chi(\xi;v(t,\om))$. We then have
\begin{equation}\label{e5.31}
\po_t(\phi\chi)+a(\xi)\cdot\nabla_x(\phi\chi)-\nabla_x\cdot(A(\xi)\nabla_x(\phi\chi))=\chi\po_t\phi\quad \text{in $\DD'(\R\X\GB_d\X\R)$}.
\end{equation}
Now,  we take the global Fourier transform in $t\in\R$ and the local Fourier transform in $\om\in\GB_d$ of the functions $\phi\chi$ and $\chi\phi_t$ to obtain  $\hat g(\tau,\k;\xi)$ for
$(\phi\chi)(t,\om,\xi)$ and $\hat h(\tau,\k;\xi)$ for $(\chi\po_t\phi)(t,\om,\xi)$ in $L^2(\R\X\GB_d\X\R)$, where $\k=(\k_1,\cdots,\k_d)$ runs along the countable set given by \eqref{e5.30}.   For example,
$$
\hat g(\tau,\k;\xi)=\int_{\R}\int_{\GB_d}(\phi\chi)(t,\om,\xi) e^{-i(\tau t+\k\cdot \om)}\,dt\,d\mm(\om),
$$
so that
$$
(\phi\chi)(t,\om,\xi)=\sum_{\k\in \GG}\int_\R\hat g(\tau,\k;\xi) e^{i(\tau t+\k\cdot\om)}\,d\tau.
$$
Here, the countable set $\GG$ is defined by
$$
\GG:=\{ 2\pi\l\,:\, \l\in \Gr(v)\}
$$
and it contains $2\pi\Sp(\chi(t,\cdot,\xi))$ for all $(t,\xi)\in\R^2$, since, for each $\xi\in\R$, $\chi(\xi;v)$ is a Borelian function of $v$. 

Taking the global Fourier transform in $t\in\R$ and the local Fourier transform in $\om\in\GB_d$ on \eqref{e5.31}, we obtain
$$
 \bigl(i(\tau+a(\xi)\cdot\k)+\k^\top A(\xi)\k\bigr)\hat g=\hat h.
 $$
As in \cite{CP2}, we introduce the parameter $\ell>0$, to be chosen later, and write
$$
\bigl(\sqrt{\ell}+i(\tau+a(\xi)\cdot\k)+\k^\top A(\xi)\k\bigr)\hat g=\hat h+\sqrt{\ell}\,\hat g.
$$ 
We then get
$$
\hat g=(\hat h+\sqrt{\ell}\,\hat g)\frac{1}{\sqrt{\ell}+i(\tau+a(\xi)\cdot\k)+\k^\top A(\xi)\k}.
$$

Integrating in $\xi$ and applying Cauchy-Schwarz inequality, we find
$$
|\widehat{\phi v}|^2(\tau,\k)\le 2\left(\int_\R |\hat h|^2\,d\xi+\ell\int_{\R}|\hat g|^2\,d\xi\right)\int_{|\xi|\le\|u_0\|}\left|\frac1{\sqrt{\ell}+i(\tau+a(\xi)\cdot\k)+\k^\top A(\xi)\k}\right|^2\,d\xi.
$$
Now, recalling \eqref{e5.14},  we are going to prove that $I_\infty=0$. 
Suppose, by contradiction that $I_\infty>0$.  We choose a finite set $F_\ve$ as in Lemma~\ref{L:5.2}~(iv), with $\ve= I_\infty/4$, and denote $\bar F=2\pi F_\ve$. 
In particular, there is $\d_0>0$ such that $|\k|\ge\d_0$ for $\k\in\bar F$, and we recall that, since $\Me(v(t,\cdot))=0$, we have
$$
\widehat{\phi v}(\tau,0)=0.
$$
 Observing that 
 $$
\left|\frac1{\sqrt{\ell}+i(\tau+a(\xi)\cdot\k)+\k^\top A(\xi)\k}\right|^2 \le\frac1{\ell+(\tau+a(\xi)\cdot\k)^2+(\k^\top A(\xi)\k)^2},
$$
condition \eqref{e1.4'}, for $|\k|>\d>0$, with $0<\d<\d_0$,  implies
 $$
 |\widehat{\phi v}|^2\le \frac{2\om_\d(\ell)}{\ell}\int_{\R} |\hat h|^2\,d\xi+2\om_\d(\ell)\int_\R|\hat g|^2\,d\xi.
 $$
 Therefore, integrating in $\tau$, summing over $\k\in \bar F$, and majorizing the right-hand side extending the summation for all $\k\in\GG$, we get
 \begin{multline}\label{e5.1000}
 \sum_{\k\in\bar F}\int_\R \widehat{\phi v}|^2\,d\tau \le  \frac{2\om_\d(\ell)}{\ell}\sum_{\k\in\GG} \int_{\R}\int_{\R}|\hat h|^2\,d\xi\,d\tau+2\om_\d(\ell)\sum_{\k\in\GG}\int_\R\int_{\R}|\hat g|^2\,d\xi\,d\tau\\
         \le  \frac{2\om_\d(\ell)}{\ell}\int\limits_{\R\X\GB_d\X\R}(\chi\phi_t)^2\,dt\,d\mm(\om)\,d\xi+2\om_\d(\ell)\int\limits_{\R\X\GB_d\X\R}|\phi\chi|^2\,dt\,d\mm(\om)\,d\xi.
  \end{multline}
Let us denote 
$$
v_0(t,x)=\sum_{\l\in F_\ve} a_\l(t) e^{2\pi i\l\cdot x},\quad\text{with}\quad a_\l(t)=\Me(e^{-2\pi i\l\cdot (\cdot)} v(t,\cdot)).
$$
Then,  proceeding the integration on $\xi$ on the right-hand side of the second inequality in \eqref{e5.1000}, observing that $|\chi|^2=|\chi|$,  using Plancherel and the definition of $v_0$ on the left-hand side of the first inequality in \eqref{e5.1000}, we get
$$
\int_{\R\X\GB_d}|\phi v_0|^2\,dt\,d\mm(\om)\le \frac{2\om_\d(\ell)}{\ell}\int\limits_{\R\X\GB_d}|\phi_t|^2|v| \,dt\,d\mm(\om)+2\om_\d(\ell)\int\limits_{\R\X\GB_d}|\phi|^2|v|\,dt\,d\mm(\om).
$$ 
Remembering that
$$
 \|v_0(t)\|_{L^2(\GB)}^2\ge \frac{3 I_\infty}4
 $$   
we get
\begin{equation}\label{e5.32}
\begin{aligned}
I_\infty\int_\R|\phi|^2\,dt &\le \frac83 \om_\d(\ell) \left(\int_{\GB_d}|v|^2\,d\mm(\om)\right)^{1/2}\left(\frac1{\ell}\int_\R|\phi_t|^2\,dt+\int_\R|\phi|^2\,dt\right)\\
         &\le \frac83 I_\infty^{1/2}\om_\d(\ell)\left(\frac1{\ell}\int_\R|\phi_t|\,dt+\int_\R|\phi|^2\,dt\right).
         \end{aligned}
\end{equation}
We now choose $\ell$ small enough so that $\frac83\om_\d(\ell)/I_\infty^{1/2}\le \frac12$ and find from \eqref{e5.32} that 
$$
I_\infty\int_\R|\phi|^2\,dt\le 2I_\infty^{1/2}\frac{\om_\d(\ell)}{\ell}\int_\R|\phi_t|^2\,dt.
$$
The conclusion is now completely identical to the one in \cite{CP2}. We choose a sequence of functions $\phi_B(t)=1$ for $|t|\le B$, with $B$ a given large number and $\phi_B(t)=\frac{2B-|t|}{B}$ for $B\le |t|\le 2B$, and $\phi_B(t)=0$ for $|t|\ge 2B$. In the above inequality, we find 
$$
I_\infty^{1/2}\le C\frac{\om_\d(\ell)}{B^2\ell},
$$
where $C>0$ is an absolute constant. When $B\to\infty$ we get that $I_\infty=0$ which contradicts the initial assumption about $I_\infty$. Hence $I_\infty=0$ and so $v(t,\om)=0$ for a.e.\ $(t,\om)\in\R\X\GB_d$. 

\end{proof} 

We then arrive at the final conclusion.

\begin{lemma}[Decay] We have
\begin{equation}\label{e5.33}
\lim_{t\to\infty}\int_{\GB_d}|u(t,\om)|\,d\mm(\om)=0.
\end{equation}
\end{lemma}

\begin{proof} From what was just proven,   we have
$$
\int_0^1\int_{\GB_d}|v^{k_j}(s,\om)|\,d\mm(\om)\,ds\to0\quad \text{as $j\to\infty$}.
$$
Therefore, taking any $t>k_j+1$,  since the $L^1(\GB_d)$-norm of $u(s,\cdot)$ in decreasing in $s$, by \eqref{eT3.2'} with $v\equiv0$, we have
$$
\int_0^1\int_{\GB_d}|v^{k_j}(s,x)|\,d\mm(\om)\,ds\ge\int_{\GB_d}|u(1+k_j,\om)|\,d\mm(\om)\ge \int_{\GB_d}|u(t,\om)|\, d\mm(\om).
$$
Hence, \eqref{e5.33} follows.

\end{proof}

\section{Acknowledgements} 

H.~Frid gratefully acknowledges the support from CNPq, through grant proc.\ 305963/2014-7, and FAPERJ, through grant proc.\ E-26/103.019/2011. 

 Y.~Li gratefully acknowledges the support from NSF of China, through grant 1183011 and Shanghai Committee of Science and Technology, through grant 15XD\-1502300.

\end{document}